\documentclass[12pt]{amsart}%
\usepackage{amsfonts}
\usepackage{amsmath}
\usepackage{amssymb}
\usepackage{graphicx}%
\makeatletter
\@addtoreset{equation}{section}
\makeatother

\newtheorem{theorem}{Theorem}[section]

\newtheorem{corollary}[theorem]{Corollary}

\newtheorem{lemma}[theorem]{Lemma}

\newtheorem{proposition}[theorem]{Proposition}
\newtheorem{remark}[theorem]{Remark}

\begin{document}
\title[Ground states of bi-harmonic equations with critical exponential growth]{Ground states of bi-harmonic equations with critical exponential growth involving constant and trapping potentials}
\author{Lu Chen, Guozhen Lu and Maochun Zhu}
\address{School of Mathematics and Statistics, Beijing Institute of Technology, Beijing 100081, P. R. China}
\email{chenlu5818804@163.com}
\address{Department of Mathematics\\
University of Connecticut\\
Storrs, CT 06269, USA}
\email{guozhen.lu@uconn.edu}
\address{Faculty of Science\\
Jiangsu University\\
Zhenjiang, 212013, P. R. China\\}
\email{zhumaochun2006@126.com}

\thanks{The research of the second author was supported partly by the Simons Foundation. The research of the
third author was supported by Natural Science Foundation of China (11601190),
Natural Science Foundation of Jiangsu Province (BK20160483) and Jiangsu
University Foundation Grant (16JDG043). }

\begin{abstract}
In this paper, we first give a necessary and sufficient condition for the
boundedness and the compactness for a class of nonlinear functionals in
$H^{2}\left(  \mathbb{R}^{4}\right)$. Using this result and the principle of symmetric criticality, we can present a relationship between the existence of the nontrivial solutions to the semilinear bi-harmonic equation of the form
\[
(-\Delta)^{2}u+\gamma u=f(u)\ \text{in}\ \mathbb{R}^{4}
\]
and the range of $\gamma\in \mathbb{R}^{+}$, where $f\left(  s\right)$ is the general nonlinear term having the critical
exponential growth at infinity. \ \ \

Though the existence of the nontrivial solutions for the bi-harmonic equation with the critical exponential growth has been studied in the literature, it seems that nothing is known so far about the existence of the ground-state solutions for this class of equations involving the trapping  potential introduced by Rabinowitz in \cite{Rabinowitz}.   Since the trapping potential is not necessarily symmetric, classical radial method cannot be applied to solve this problem. In order to overcome this difficulty, we first establish the existence of the ground-state solutions for the equation
\begin{equation}\label{con}
(-\Delta)^{2}u+V(x)u=\lambda s\exp(2|s|^{2}))\ \text{in}\ \mathbb{R}^{4},
\end{equation}
when $V(x)$ is a positive constant using  the Fourier rearrangement and the Pohozaev identity. Then we will explore the relationship between the Nehari manifold and the corresponding limiting Nehari manifold to derive the existence of the ground state solutions for the equation \eqref{con} when $V(x)$ is the Rabinowitz type trapping  potential, namely it satisfies
$$
0<V_{0}=\underset{x\in\mathbb{R}^{4}}{\inf}V\left(  x\right)
<\underset{\left\vert x\right\vert \rightarrow\infty}{\lim}V\left(  x\right)
<+\infty.
$$
The same result and proof applies to the harmonic equation with the critical exponential growth involving the Rabinowitz type trapping potential in $\mathbb{R}^2$.
\end{abstract}

\maketitle {\small {\bf Keywords:} Rabinowitz potential, Ground state
solutions; Bi-harmonic equations;  Adams' inequalities. \\

{\bf 2010 MSC.} Primary 46E35; 35J91;
Secondary   26D10.}

\section{\bigskip Introduction}

Let $\Omega$ be an open domain in $\mathbb{R}^n$. We will consider the following nonlinear partial differential equation with
critical growth%

\begin{equation}
\left(  -\Delta\right)  ^{m}u=f\left(  u\right)  \text{ in }\Omega\subset%
\mathbb{R}
^{n}, \label{eq}%
\end{equation}
where $m=1$ or $2$. Equations (\ref{eq}) have been extensively studied by many
authors in bounded\ and unbounded domains.

In the case $n>2m$, the subcritical and critical growth means that the
nonlinearity cannot exceed the polynomial of\textit{\ }degree $\frac{n+2m}%
{n-2m}$ by the Sobolev embedding. While in the case $n=2m$, we say that
$f\left(  s\right)  $ has \textit{critical exponential growth} at infinity if
there exists $\alpha_{0}>0$ such that \
\begin{equation}
\underset{t\rightarrow\infty}{\lim}\frac{f\left(  t\right)  }{\exp\left(
\alpha t^{2}\right)  }=%
\genfrac{\{}{.}{0pt}{}{0\text{, \ \ for }\alpha>\alpha_{0}}{+\infty\text{, for
}\alpha<\alpha_{0}}
\label{exponential critical}%
\end{equation}

\bigskip The critical exponential growth in the case $m=1,n=2$ is given
by\ the Trudinger-Moser inequality (\cite{Mo}, \cite{Tru}):%
\begin{equation}
\underset{\left\Vert \nabla u\right\Vert _{L^{2}\left(  \Omega\right)  }\leq
1}{\underset{u\in H_{0}^{1}\left(  \Omega\right)  }{\sup}}\int_{\Omega
}e^{\alpha\left\vert u\right\vert ^{2}}dx<\infty\text{ iff }\alpha\leq4\pi,
\label{moser}%
\end{equation}
and in the case $m=2,n=4$\ is given by the Adams inequality \cite{A}:%
\[
\underset{\left\Vert \Delta u\right\Vert _{L^{2}\left(  \Omega\right)  }\leq
1}{\underset{u\in H_{0}^{2}\left(  \Omega\right)  }{\sup}}\int_{\Omega
}e^{\alpha\left\vert u\right\vert ^{2}}dx<\infty\text{ iff }\alpha\leq
32\pi^{2}.
\]

The study of equation (\ref{eq}) with the critical exponential growth on
bounded domain also involves a lack of compactness similar to the case $n>2m$
at certain levels that the Palais-Smale compactness condition fails due
to concentration phenomena. In order to overcome the possible failure of the Palais-Smale compactness condition, there is a common approach involved with the Trudinger-Moser and Adams type inequalities (see \cite{Brezis}, \cite{Ambrosetti}%
,\ \cite{de Figueiredo}, \cite{Adimurthi}, \cite{Lam} and references therein).
\vskip0.1cm

If $\Omega$ is the entire Euclidean space $\mathbb{R}^n$, the earlier study of the existence of solutions for equation (\ref{eq}) with the critical exponential growth can date back to the work of Atkinson and Peletier \cite{Atkinson,Atkinson1}. Indeed, the authors obtained the existence of ground state solutions for equation (\ref{eq}) by assuming that there exists some $y_{0}>0$ such that
$g\left(  t\right)  =\log f\left(  t\right)  $ satisfies $$g^{\prime}\left(
t\right)  >0,g^{\prime\prime}\left(  t\right)  \geq0,$$ for any $t\geq y_{0}$. This kind of growth condition allows us to take the nonlinearity $f\left(
t\right)  =\left(  t^{2}-t\right)  \exp\left(  t^{2}\right), $ which has critical exponential growth. In the literature, many authors have
 considered the existence of solutions for equations of the form
\begin{equation}
\left(  -\Delta\right)  ^{m}u+V\left(  x\right)  u=f\left(  u\right)  \text{
in }%
\mathbb{R}
^{n},\label{potential eq}%
\end{equation}
where $n=2m$, the nonlinearity $f\left(  s\right)  $ has critical exponential
growth, and the potential $V(x)$ is bounded away from zero. For the equation
(\ref{potential eq}), the loss of compactness may be produced not only by the
concentration phenomena but also by the vanishing phenomena! We will describe some of the relevant works below.

\vskip0.1cm

When $V\left(  x\right)  $\ is a\ \bigskip coercive potential, that is,
\[
V\left(  x\right)  \geq V_{0}>0,\text{and additionally either}\underset{x\rightarrow\infty
}{\lim}V(x)=+\infty\text{ or }\frac{1}{V}\in L^{1}\left(
\mathbb{R}
^{n}\right)  ,
\]
the existence and multiplicity results of equation (\ref{potential eq}) can be
found in the papers \cite{Yang},\cite{Lamlu},\cite{zhao} and the
references\ therein. Their proofs depend crucially on the compact embeddings
given by the coercive potential, and the vanishing phenomena can be ruled out.
\vskip0.1cm

When $V\left(  x\right)  $ is the constant potential, i.e. $V\left(  x\right)
=V_{0}>0$, the natural space for a variational treatment of (\ref{potential eq})
is $H^{m}\left(
\mathbb{R}
^{n}\right)  $. It is well known that the embedding $H^{m}\left(
\mathbb{R}
^{n}\right)  \hookrightarrow L^{2}\left(
\mathbb{R}
^{n}\right)  $\ is continuous but not compact, even in the radial case. In
the case $m=2$, the existence of nontrivial solutions for equation
(\ref{potential eq}) was obtained by Chen et al in \cite{chenluzhang} (see also
\cite{bao})\ (see \cite{Alves}  for $m=1$) under the assumptions that for any
$p>2$,
\begin{equation}
f\left(  s\right)  \geq\eta_{p}s^{p-1},\forall s\geq0, \label{polo}%
\end{equation}
where $\eta_{p}$ is some constant depending on $p$, and by Sani in \cite{sani}
(see \cite{Rufs} or \cite{Atkinson} for $m=1$) under the assumption \
\begin{equation}
\underset{s\rightarrow+\infty}{\lim}\frac{sf\left(  s\right)  }{\exp\left(
32\pi^{2}s^{2}\right)  }\geq\beta_{0}>0. \label{exp}%
\end{equation}

\bigskip In their proofs, the so-called Trudinger-Moser-Adams inequality in
the whole $%
\mathbb{R}
^{2}$ or $
\mathbb{R}^{4}$ plays a crucial role. \ Now, let's mention some of these\ inequalities.
In 2000, Adachi-Tanaka \cite{Adachi-Tanaka}  (see also do O \cite{Do}) obtained a sharp Trudinger-Moser inequality
on $%
\mathbb{R}
^{n}$:\textit{ }%
\begin{equation}
\underset{\int_{\mathbb{R}^{n}}\left\vert \nabla u\right\vert ^{n}dx\leq
1}{\underset{u\in{W^{1,n}(\mathbb{R}^{n})},}{\sup}}\int_{\mathbb{R}^{n}}%
\Phi_{n}(\alpha|u|^{\frac{n}{n-1}})dx\leq{C(\alpha,n)\Vert u\Vert_{n}^{n}%
},\text{ iff }0<\alpha<\alpha_{n}, \label{Adachi-Tanaka}%
\end{equation}
where\textit{ }$\Phi_{n}(t):=e^{t}-\sum_{i=0}^{n-2}\frac{t^{i}}{i!}$\textit{.
}Note that  the inequality \eqref{Adachi-Tanaka}  has
the subcritical form, that is $\alpha<\alpha_{n}$. Later, in \cite{ruf} and
\cite{liruf}, Li and Ruf showed that the best exponent $\alpha_{n}$ becomes
admissible if the Dirichlet norm $\int_{\mathbb{R}^{n}}\left\vert \nabla
u\right\vert ^{2}dx$ is replaced by Sobolev norm $\int_{\mathbb{R}^{n}}\left(
\left\vert u\right\vert ^{2}+\left\vert \nabla u\right\vert ^{2}\right)  dx$.
More precisely, they proved that%

\begin{equation}
\underset{\int_{\mathbb{R}^{n}}\left(  \left\vert u\right\vert ^{n}+\left\vert
\nabla u\right\vert ^{n}\right)  dx\leq1}{\underset{u\in W^{1,n}\left(
\mathbb{R}
^{n}\right)  }{\sup}}\int_{%
\mathbb{R}
^{2}}\Phi_{n}\left(  \alpha\left\vert u\right\vert ^{\frac{n}{n-1}}\right)
dx<+\infty,\text{ iff }\alpha\leq\alpha_{n}. \label{Li-ruf}%
\end{equation}
The proofs of both the critical and subcritical Trudinger-Moser inequalities
\eqref{Adachi-Tanaka} and \eqref{Li-ruf} rely on the P\'{o}lya-Szeg\"{o} inequality and the symmetrization argument.
Lam and Lu (\cite{LamLu-AIM}, \cite{LaLu4}) developed a symmetrization-free method to establish the critical Trudinger-Moser inequality
(see also Lam, Lu and Tang \cite{LamLuTang} for a proof of the subcritical Trudinger-Moser inequality) in settings such as the Heisenberg group  where the P\'{o}lya-Szeg\"{o} inequality fails. Such an argument also provides an alternative proof of both critical and subcritical Trudinger-Moser inequalities \eqref{Adachi-Tanaka} and \eqref{Li-ruf} in the Euclidean space. In fact, the equivalence and relationship between the supremums of critical and subcritical Trudinger-Moser inequalities have been established by Lam, Lu and Zhang \cite{lamluzhang1}.
\vskip0.1cm

In 1995, Ozawa~\cite{O} obtained  the Adams inequality in Sobolev
space $W^{m,\frac{n}{m}}(\mathbb{R}^{n})$ on the entire Euclidean space
$\mathbb{R}^{n}$ by using the restriction $\Vert\Delta^{\frac{m}{2}}%
u\Vert_{\frac{n}{m}}\leq{1}$. However, with the argument in \cite{O,KSW}, one
cannot obtain the best possible exponent $\beta$ for this type of
inequality. Sharp Adams inequality in the case of even order of derivatives was proved by Ruf and Sani ~\cite{RS}
under the constraint
$$\{
u\in W^{{m},\frac{{n}}{m}}| \Vert(I-\Delta)^{\frac{m}{2}}\Vert_{\frac{n}{m}%
}\leq1
\},$$
when $m$ is an even integer. When  the order $m$ of the derivatives is odd, a sharp Adams inequality was established
by
Lam and Lu \cite{LamLu-JDE2012}. A uniform proof was given for all orders $m$ of derivatives including fractional orders of derivatives by Lam and Lu in ~\cite{LaLu4} through a rearrangement-free argument.
\vskip0.1cm

  The authors in \cite{LaLu4} obtained
the sharp Adams inequality under the Sobolev norm\ constraint:%

\begin{equation}
\underset{\Vert\Delta u\Vert_{2}^{2}+\Vert u\Vert_{2}^{2}\leq{1}}%
{\underset{u\in H{^{2}(\mathbb{R}^{4})}}{\sup}}\int_{{\mathbb{R}^{4}}}\left(
\exp(\beta|u(x)|^{2})-1\right)  dx\left\{
\begin{array}
[c]{l}%
\leq C\text{ if }\beta\leq{32\pi}^{2}{,}\\
=+\infty\text{ \ if }\beta>{32\pi}^{2}{.}%
\end{array}
\right.  \label{Adams entire space}%
\end{equation}

\bigskip In 2011, Ibrahim et al \cite{IMN} discovered a sharpened
Trudinger-Moser inequality on $%
\mathbb{R}
^{2}$--the Trudinger-Moser inequality\ with the exact growth condition:
\begin{equation}
\underset{\int_{\mathbb{R}^{4}}|\nabla u|^{2}dx\leq1}{\underset{u\in
H^{1}\left(
\mathbb{R}
^{2}\right)  }{\sup}}\int_{\mathbb{R}^{2}}\frac{\exp(4\pi|u|^{2}%
)-1}{(1+|u|)^{p}}dx\leq C_{p}\int_{\mathbb{R}^{2}}|u|^{2}dx\ \text{iff }%
p\geq2. \label{exact}%
\end{equation}
Later, (\ref{exact}) was extended to the general case $n\geq3$\ by Masmoudi
and Sani \cite{MS2} (see Lam et al \cite{Lamluzhang} for  inequalities with exact growth under different norms) and to the
framework of hyperbolic space by Lu and Tang in \cite{LuTa2}. It is
interesting to notice that the Trudinger-Moser inequality\ with the exact
growth condition can imply both the inequalities (\ref{Adachi-Tanaka})
and\ (\ref{Li-ruf}).
\vskip0.1cm

The Adams' inequality\ with the exact growth condition was obtained by
Masmoudi and Sani \cite{MS} in dimension $4$: \ \
\begin{equation}
\underset{\int_{\mathbb{R}^{4}}|\Delta v|^{2}dx\leq1}{\underset{u\in
H^{2}\left(
\mathbb{R}
^{4}\right)  }{\sup}}\int_{\mathbb{R}^{4}}\frac{\exp(32\pi^{2}|v|^{2}%
)-1}{(1+|v|)^{p}}dx\leq C_{p}\int_{\mathbb{R}^{4}}|v|^{2}dx\ \text{iff }%
p\geq2, \label{int1}%
\end{equation}
and then established in any dimension $n\geq3$ by Lu et al in \cite{LuTZ} (see
\cite{MS1} for higher order case). Further improvement of Adams inequalities can also be found in recent work 
of Lu and Yang \cite{LuYang} where sharpened Hardy-Adams inequalities were established in $\mathbb{R}^4$ using Fourier analysis on hyperbolic spaces (see also \cite{LiLuYang} for higher even dimensions).

\bigskip Based on the Trudinger--Moser inequality with the exact growth,
Ibrahim et al obtained a sufficient and necessary condition for compactness of
general nonlinear functionals (see \cite{MS2} for $n\geq3$). This sufficient
and necessary condition is a strong tool to study the
existence of solutions for the semilinear equation under a very general assumption on the nonlinearity.
Indeed, they consider the equations of the form:%
\begin{equation}
-\Delta u+\gamma u=f\left(  u\right)  \text{ in }%
\mathbb{R}
^{2}, \label{equa}%
\end{equation}
where $\gamma$ is a positive constant and $f\left(  s\right)  $ has the critical exponential growth at infinity.
They establish the following result.

\begin{proposition}
\label{sanire}If  $f$ satisfies $f(0)=0$ and the conditions
(\ref{exponential critical}), (i) and (ii)(see Section 2), then\ there exists
$\gamma^{\ast}\in\left(  0,+\infty\right)  $ such that for each $\gamma
\in\left(  0,\gamma^{\ast}\right)  $, the equation admits a positive ground
state solution.
\end{proposition}

The number $\gamma^{\ast}$ above is associated with the so called
Trudinger--Moser ratio, and both the growth conditions (\ref{polo}) and
(\ref{exp}) imply that the constant $\gamma$ appearing in (\ref{equa})
satisfies
\[
\gamma<\gamma^{\ast}.
\]
Their arguments depend crucially on the Pohozaev identity\ and Schwarz
symmetrization argument.

\bigskip Motivated by the results just described, in this paper, we are
interested in\ the existence of ground state solutions for the biharmonic equations%

\begin{equation}
(-\Delta)^{2}u+V(x)u=f(u)\ \text{in}\ \mathbb{R}^{4},\label{b1-2}%
\end{equation}
where the nonlinear term $f\left(  s\right)  $ has the critical exponential
growth (\ref{exponential critical}) at infinity and the potential $V(x)$ is
bounded away from zero. \
\vskip0.1cm

As far as we know, there are no related results about the existence of
\textit{ground state solutions} for the problem (\ref{b1-2}) by means of
variational methods. In study of the problems involving the bi-harmonic
operator, we will encounter many more difficulties than in the case for the Lapacian. For example, we cannot always rely on the maximum
principle, and there is no P\'{o}lya-Szeg\"{o} type inequality for the second order derivatives. Thus, we
cannot use Schwarz symmetrization principle. Moreover, if $u$
belongs $H^{2}\left(
\mathbb{R}
^{4}\right)  $, we cannot claim that \textit{\ }$\left\vert u\right\vert
$,$\ u^{+}$ or $u^{-}$ belong to $H^{2}\left(
\mathbb{R}
^{4}\right)  $. Therefore, we cannot expect to obtain a positive
solution.\textit{\ }

\section{The main results}

In order to obtain the existence of solutions to the equation (\ref{b1-2}),\ we first
establish the necessary and sufficient conditions for the boundedness and the
compactness of general nonlinear functionals in $H^{2}(\mathbb{R}^{4})$.

\begin{theorem}
[Boundedness]\label{thm1} Suppose that $g:\mathbb{R}\rightarrow\lbrack
0,+\infty)$ is a Borel function and define
\[
G(u)=\int_{\mathbb{R}^{4}}g(u)dx.
\]
Then for any $K>0$, the following conditions are equivalent

(1)$\ \lim\limits_{t\rightarrow+\infty}|t|^{2}\exp(-\frac{1}{K}|t|^{2}%
)g(t)<\infty,\lim\limits_{t\rightarrow0}|t|^{-2}g(t)<\infty.$

(2) There exists a constant $C_{g,K}>0$ such that for any $u\in H^{2}%
(\mathbb{R}^{4})$ satisfying $\int_{\mathbb{R}^{4}}|\Delta u|^{2}dx\leq
32\pi^{2}K$, there holds
\[
\int_{\mathbb{R}^{4}}g(u)dx\leq C_{g,K}\int_{\mathbb{R}^{4}}|u|^{2}dx.
\]

\end{theorem}

\begin{theorem}
[Compactness]\label{thm2} Suppose that $g:\mathbb{R}\rightarrow\lbrack
0,+\infty)$ is a continuous function and define
\[
G(u)=\int_{\mathbb{R}^{4}}g(u)dx.
\]
Then for any $K>0$, the following conditions are equivalent

(3)$\ \lim\limits_{t\rightarrow+\infty}|t|^{2}\exp(-\frac{1}{K}|t|^{2}%
)g(t)=0,\ \lim\limits_{t\rightarrow0}|t|^{-2}g(t)=0.$

(4) For any radially symmetric\ sequence $\{u_{k}\}_{k}\in H^{2}%
(\mathbb{R}^{4})$ satisfying $\int_{\mathbb{R}^{4}}|\Delta u|^{2}dx\leq
32\pi^{2}K$ and weakly converging to some $u$, we have that $G(u_{k}%
)\rightarrow G(u)$.
\end{theorem}

As an application, we study the following bi-harmonic equation with the constant
potential,
\begin{equation}
(-\Delta)^{2}u+\gamma u=f(u)\ \text{in}~\mathbb{R}^{4},\ \label{bi-harmonic1}%
\end{equation}
where the nonlinearity $f(t)$ is a continuous function on $\mathbb{R}$
satisfying (\ref{exponential critical}), $f(0)=0$ and the following properties:

(i) (Ambrosetti-Rabinowitz condition) There exists $\mu>2$ such that $0<\mu F(t)=\mu\int_{0}^{t}f(s)ds\leq tf(t)$
for any $t\in\mathbb{R}^{+}$;\newline

(ii) There exist $t_{0}$ and $M_{0}>0$ such that $F(t)\leq M_{0}f(t)$ for any
$t\geq t_{0}$.\newline

\begin{theorem}
\label{thm3} Assume that $f$ satisfies $f(0)=0$ and the conditions
(\ref{exponential critical}), (i) and (ii), then there exists $\gamma^{\ast
}\in(0,+\infty]$ such that for any $\gamma\in(0,\gamma^{\ast})$, the equation
(\ref{bi-harmonic1}) admits a non-trivial radial solution. Moreover,
$\gamma^{\ast}$ is equal to the Admas ratio:
\[
C_{A}^{\ast}=\sup\left\{  \frac{2}{\Vert u\Vert_{2}^{2}}\int_{\mathbb{R}^{4}%
}F(u)dx|\ u\in H_{r}^{2}(\mathbb{R}^{4}),\Vert\Delta u\Vert_{2}^{2}\leq
\frac{32\pi^{2}}{\alpha_{0}}\right\}  ,
\]
where $H_{r}^{2}(\mathbb{R}^{4})$  is the collection of all radial functions
in $H^{2}(\mathbb{R}^{4})$. In particular, $\gamma^{\ast}=+\infty$ is equivalent to
\[
\lim_{t\rightarrow\infty}\frac{t^{2}F(t)}{\exp(\alpha_{0}t^{2})}=+\infty.
\]

\end{theorem}

\begin{remark}
If $F(t)=\frac{\exp(t^{2})-1-t^{2}}{1+|t|^{\theta}}$, obviously,
$f(t)=F^{\prime}(t)$ satisfies the conditions (\ref{exponential critical}),
(i) and (ii). By the Adams inequality with exact growth (\ref{int1}), we know
that if $\theta<2$, then $\gamma^{\ast}=+\infty$, and the equation
(\ref{bi-harmonic1}) admits a non-trivial radial solution for any $\gamma>0$.
If $\theta\geq2$, then $\gamma^{\ast}<+\infty$, and the equation
(\ref{bi-harmonic1}) admits a non-trivial radial solution for any $\gamma
\in(0,\gamma^{\ast})$. Both the growth conditions\ of the nonlinearities used
in \cite{sani} and \cite{chenluzhang} imply that the constant $\gamma$
appearing in Theorem \ref{thm3} satisfies
\[
\gamma<\gamma^{\ast}.
\]

\end{remark}

\begin{corollary}
\label{coroll}Assume that $f$ satisfies\ $f(0)=0$ and the conditions
(\ref{exponential critical}), (i), (ii), and $F(t)$ satisfies
\[
\lim_{t\rightarrow\infty}\frac{t^{2}F(t)}{\exp(\alpha_{0}t^{2})}=+\infty,
\]
then for any $\gamma>0$, the equation (\ref{bi-harmonic1}) admits a
non-trivial radial solution.
\end{corollary}

\bigskip Until now, whether the solutions obtained in Theorem
\ref{thm3} and Corollary \ref{coroll} are ground state solutions is unknown. However, if the
nonlinearity has the special form $f(t)=\lambda t\exp(2|t|^{2})$, we can prove that the solutions obtained
are ground-state solutions.
\begin{theorem}
\bigskip\label{thm4} For any $\gamma\in(0,+\infty)$, the equation
\end{theorem}%
\[
(-\Delta)^{2}u+\gamma u=\lambda u\exp(2|u|^{2})\ \text{in}~\mathbb{R}^{4}%
\]
admits a radial ground state solution if $\lambda\in\left(  0,\gamma\right)  $.
\vskip0.1cm

Based on the above Theorem, we can also obtain the existence of ground state
solutions of the bi-harmonic equation with the Rabinowitz type potential, that is,
the potential $V\left(  x\right)  $ is a continuous function satisfying
\begin{equation}
0<\lambda<V_{0}=\underset{x\in\mathbb{R}^{4}}{\inf}V\left(  x\right)
<\underset{\left\vert x\right\vert \rightarrow\infty}{\lim}V\left(  x\right)
=\gamma<+\infty.\label{Rabinowitz}%
\end{equation}
This kind of potential was first introduced by Rabinowitz in
\cite{Rabinowitz}.

\begin{theorem}
\bigskip\label{thm5} Assume that $V\left(  x\right)  $ is a continuous
function satisfying\ (\ref{Rabinowitz}), the equation
\begin{equation}
(-\Delta)^{2}u+V\left(  x\right)  u=\lambda u\exp(2|u|^{2})\ \text{in}%
~\mathbb{R}^{4}\ \ \label{potential}%
\end{equation}
admits a non-radial ground state solution.
\end{theorem}

\bigskip In general, the ground state solution can be constructed by showing
that the infimum on the Pohozaev or Nehari manifold is achieved. This is
equivalent to proving that the mountain-pass minimax level is
achieved (see \cite{Ambrosetti},\cite{Rabinowitz1},\cite{Rabinowitz2}). In the
proof of Theorem \ref{thm4},\ we cannot use the Schwarz symmetrization principle directly. In order to overcome this difficulty, we will apply
the Fourier rearrangement proved by Lenzmann and Sok in \cite{Lenzmann} to obtain a radially minimizing sequence for the
infimum on the Pohozaev manifold. While in the proof of Theorem
\ref{thm5},\ we will exploit the relationship between the Nehari manifold and
the corresponding limiting Nehari manifold.
\vskip0.1cm

Using Proposition \ref{sanire} and carrying out the same proof procedure of
Theorem \ref{thm5}, we can also obtain the existence of ground state solutions
of the following Laplacian equation with the Rabinowitz type potential introduced in \cite{Rabinowitz}:
\begin{equation}
-\Delta u+V\left(  x\right)  u=\lambda u\exp(|u|^{2})\ \text{in}%
~\mathbb{R}^{2}.\ \ \label{laplacian}%
\end{equation}

\begin{theorem}
\bigskip\label{thm6} Assume that $V\left(  x\right)  $ is a continuous
function satisfying
\[
0<\lambda<V_{0}=\underset{x\in\mathbb{R}^{2}}{\inf}V\left(  x\right)
<\underset{\left\vert x\right\vert \rightarrow\infty}{\lim}V\left(  x\right)
=\gamma<+\infty,
\]
the equation (\ref{laplacian}) admits a non-radial ground state solution.
\end{theorem}
\vskip0.1cm

As far as we know, the Rabinowitz type potentials are only involved in the
study of equations with the subcritical polynomial growth  (see e.g., \cite{Lions}, \cite{Rabinowitz} and \cite{XWang}). In
the case of $m=1,n=2$, when we replace the operator $-\Delta$ by
$-\varepsilon^{2}\Delta$ in the above theorem when the nonlinear term has the exponential growth, the existence of semiclassical
state $u_{\varepsilon}$ was obtained by\ Alves and Figueiredo in $\mathbb{R}^2$ 
\cite{Alves-f} if $\varepsilon<<1$. For other related work on the semiclassical state of nonlinear Schr\"{o}dinger equations in the case of subcritical nonlinear polynomial growth, we just name a few among a vast literature, e.g.,  \cite{Ambrosetti1, Ambrosetti2, DingNi, KangWei, LuWei, XWang}, the book \cite{AmbrosettiMalchiodi}  and many references therein.
Nevertheless, as far as we are concerned, nothing is known if $\varepsilon=1$ and the nonlinear term has the exponential growth.
Theorem \ref{thm6} appears to be the first existence result for equation with the critical exponential growth involving
the Rabinowitz type trapping potential.
\vskip0.1cm

This paper is organized as follows. Section 3 is devoted to the proof of the
necessary and sufficient conditions for the boundedness and the compactness of
general nonlinear functionals in $H^{2}(\mathbb{R}^{4})$. In Section 4, we
will prove the existence of non-trivial solutions of the equation
(\ref{bi-harmonic1}) with the constant potential under a very general assumption
on the nonlinearity. In Section 5, we prove the existence of ground state
solutions for the bi-harmonic equation (\ref{bi-harmonic1}) with the constant
potential when the nonlinearity has the special form $f(s)=\lambda s\exp(2s^{2})$. In
Section 6, we prove the existence of ground state solutions for the
bi-harmonic equation (\ref{bi-harmonic1}) with the Rabinowitz type potential.

Throughout this paper, the letter $c$ always denotes some positive constant
which may vary from line to line.

\section{ Necessary and sufficient conditions for the boundedness and
compactness}

In this section, we will give the necessary and sufficient conditions for the
boundedness and the compactness of general nonlinear functionals.

\begin{proof}
[Proof of Theorems \ref{thm1} and \ref{thm2}]\textit{ }\emph{\medskip
Necessity of (1) in Theorem \ref{thm1} and (3) in Theorem \ref{thm2}: }
\vskip0.1cm

In order to prove the necessity of (1), we only need to verify that
if (1) fails, then there exists a sequence $\{u_{k}\}_{k}\in
H^{2}(\mathbb{R}^{4})$ satisfying
\[
\int_{\mathbb{R}^{4}}|\Delta u_{k}|^{2}dx\leq32\pi^{2}K,\ \ \int
_{\mathbb{R}^{4}}|u_{k}|^{2}dx\rightarrow0,\ \ G(u_{k})\rightarrow\infty.
\]
Similarly, in order to prove the necessity of (3), we only need to show that
if (3) fails, then there exists a radially symmetric sequence ~$\{u_{k}%
\}_{k}\in H^{2}(\mathbb{R}^{4})$ satisfying $\Vert\Delta u_{k}\Vert_{2}%
^{2}\leq32\pi^{2}K$ and weakly converging to $0$, such that $G(u_{k})>\delta$
for some $\delta>0$.

First, we consider the case that the conditions (1) and
(3) fails at the origin. Let$\{\phi_{k}\}_{k}\in H^{2}(\mathbb{R}^{4})$ be a
sequence of spherically symmetric functions given by
\[
\phi_{k}(x)=%
\begin{cases}
a_{k}, & \text{if $0\leq|x|\leq R_{k}$,}\\
a_{k}(1-R_{k}^{2}-|x|^{2}+2R_{k}|x|), & \text{if $R_{k}<|x|\leq R_{k}+1$,}\\
\text{$\eta_{k}(x),$} & \text{if $|x|>R_{k}+1$,}%
\end{cases}
\]
where $\eta_{k}$ is a smooth function satisfying
\[
\eta_{k}(x)|_{\partial B_{R_{k}+1}}=0,\ \ \frac{\partial\eta_{k}(x)}%
{\partial\nu}|_{\partial B_{R_{k}+1}}=-2a_{k},
\]
and
\[
\eta_{k}(x)|_{\partial B_{R_{k}+2}}=0,\ \ \frac{\partial\eta_{k}(x)}%
{\partial\nu}|_{\partial B_{R_{k}+2}}=0.
\]
Furthermore, we assume that $\{a_{k}\}_{k}$ and $\{R_{k}\}_{k}$ are positive sequences satisfying $\lim\limits_{k\rightarrow\infty}a_{k}=0$,
and $\lim\limits_{k\rightarrow\infty}R_{k}=\infty$.

Direct calculations show that there exists a constant $c>0$ such that
\[
\int_{\mathbb{R}^{4}}|\phi_{k}|^{2}dx\leq ca_{k}^{2}R_{k}^{4},\int
_{\mathbb{R}^{4}}|\Delta\phi_{k}|^{2}dx\leq ca_{k}^{2}R_{k}^{3},\text{ and
}G(\phi_{k})\geq\frac{\omega_{3}}{4}g(a_{k})R_{k}^{4}.
\]

If (1) is violated by $\lim\limits_{t\rightarrow0}|t|^{-2}g(t)<\infty$, then
there exists a sequence $c_{k}\rightarrow\infty$ such that $g(a_{k})\geq
c_{k}a_{k}^{2}$. Let $R_{k}=a_{k}^{-1/4}+a_{k}^{-1/2}c_{k}^{-1/8}$, then
\[
a_{k}^{2}R_{k}^{4}\rightarrow0,G(u_{k})\geq\frac{\omega_{3}}{4}c_{k}a_{k}%
^{2}R_{k}^{4}\rightarrow\infty.
\]
If (3) is violated by $\lim\limits_{t\rightarrow0}|t|^{-2}g(t)>0$, then there
exists $\delta>0$ such that $g(a_{k})\geq\delta a_{k}^{2}$. Pick $R_{k}%
=a_{k}^{-1/2}$, then $a_{k}^{2}R_{k}^{4}=1$, $a_{k}^{2}R_{k}^{3}\rightarrow0$
and
\[
G(u_{k})\geq\frac{\omega_{3}}{4}g(a_{k})R_{k}^{4}\geq\frac{\omega_{3}}%
{4}\delta>0.
\]
\medskip

What left is to consider the case when the conditions (1) and (3) do not hold at
infinity. Let $\{b_{k}\}_{k}\subset\mathbb{R}^{+}$, $b_{k}\rightarrow\infty$,
be such that
\[
\lim_{t\rightarrow+\infty}|t|^{2}\exp\left(  -\frac{1}{K}|t|^{2}\right)
g(t)=\lim_{k\rightarrow\infty}c_{k},
\]
where
\[
c_{k}:=b_{k}^{2}\exp\left(  -\frac{1}{K}b_{k}^{2}\right)  g(b_{k}).
\]
Set $R_{k}=\exp(-\frac{1}{K}b_{k}^{2})$, then~$c_{k}=b_{k}^{2}R_{k}g(b_{k})$.

Now, we consider the so-called Moser's sequence $\{\psi_{k}\}_{k}\in
H^{2}(\mathbb{R}^{4})$ consisting of spherically symmetric functions defined
by
\[
\psi_{k}\left(  x\right)  =%
\begin{cases}
b_{k}-\frac{2K|x|^{2}}{R_{k}^{1/2}b_{k}}+\frac{2K}{b_{k}}, &
\text{if\ $\ 0\leq|x|\leq R_{k}^{1/4}$,}\\
\frac{4K\left\vert \log\left\vert x\right\vert \right\vert }{b_{k}}, &
\text{if \ $R_{k}^{1/4}<|x|\leq1$,}\\
\eta_{k}, & \text{if \ $|x|>1$,}%
\end{cases}
\]
where $\eta_{k}$ is a smooth function satisfying
\[
\eta_{k}(x)|_{\partial B_{1}}=0,\ \ \frac{\partial\eta_{k}(x)}{\partial\nu
}|_{\partial B_{1}}=\frac{4K}{b_{k}},
\]
and
\[
\eta_{k}(x)|_{\partial B_{2}}=0,\ \ \frac{\partial\eta_{k}(x)}{\partial\nu
}|_{\partial B_{2}}=0.
\]
Careful computations yield that there exists a constant $c>0$ such that
\[
\int_{\mathbb{R}^{4}}|\psi_{k}|^{2}dx\leq\frac{cK^{2}}{b_{k}^{2}}%
,\int_{\mathbb{R}^{4}}|\Delta\phi_{k}|^{2}dx=32\pi^{2}K+O(\frac{1}{b_{k}^{2}%
}),
\]
and \
\[
G(\psi_{k})\geq\frac{\omega_{3}}{4}g(b_{k})R_{k}=\frac{\omega_{3}}{4}%
\frac{c_{k}}{b_{k}^{2}}.
\]
Define a new sequence $\{u_{k}\}_{k}\in H^{2}(\mathbb{R}^{4})$ by
$u_{k}(x)=\psi_{k}(x/S_{k})$, then
\[
\int_{\mathbb{R}^{4}}|u_{k}|^{2}dx\leq\frac{cS_{k}^{4}K^{2}}{b_{k}^{2}}%
,\int_{\mathbb{R}^{4}}|\Delta u_{k}|^{2}dx=32\pi^{2}K+O(\frac{1}{b_{k}^{2}})
\]
and
\[
G(u_{k})=S_{k}^{4}G(\psi_{k})\geq\frac{\omega_{3}}{4}\frac{{S_{k}^{4}}c_{k}%
}{b_{k}^{2}}.
\]
Assume that the condition (1) does not hold at infinity, namely
\[
\lim_{t\rightarrow+\infty}|t|^{2}\exp\left(  -\frac{1}{K}|t|^{2}\right)
g(t)=\lim_{k\rightarrow\infty}c_{k}=\infty.
\]

Set~$S_{k}^{4}=b_{k}^{2}c_{k}^{-1/2}$, there holds
\[
\int_{\mathbb{R}^{4}}|u_{k}|^{2}dx\leq\frac{cS_{k}^{4}K^{2}}{b_{k}^{2}%
}\rightarrow0,\text{ and}\ G(u_{k})\geq\frac{\omega_{3}}{4}\frac{S_{k}%
^{4}c_{k}}{b_{k}^{2}}\rightarrow\infty.
\]
Assume that the condition (3) fails at infinity, namely, there exist some
$\delta>0$ such that
\[
\lim_{t\rightarrow+\infty}|t|^{2}\exp\left(  -\frac{1}{K}|t|^{2}\right)
g(t)=\lim_{k\rightarrow\infty}c_{k}=\delta>0.
\]
Set~$S_{k}^{4}=b_{k}^{2}$, we can easily verify that $u_{k}\rightarrow0$ a.e.
$\mathbb{R}^{4}$, and
\[
\int_{\mathbb{R}^{4}}|u_{k}|^{2}dx\leq cK^{2},\int_{\mathbb{R}^{4}%
}|\Delta u_{k}|^{2}dx=32\pi^{2}K+O(\frac{1}{b_{k}^{2}}).
\]
Moreover, we also have $u_{k}\rightarrow0$ a.e. $\mathbb{R}^{4}$ and
\[
G(u_{k})\geq\frac{\omega_{3}}{4}\frac{S_{k}^{4}\delta}{b_{k}^{2}}=\frac
{\omega_{3}\delta}{4}>0.
\]
This accomplish the proof of the necessity of (1) and (3).
\vskip0.1cm

\emph{Sufficiency of (1) and (2):}

\vskip0.1cm
We first prove that (1) can imply (2).
Define a new Borel measurable function $\tilde{g}(t)$ by $\tilde
{g}(t)=g((32\pi^{2}K)^{\frac{1}{2}}t)$. Obviously,
\[
\ \lim_{t\rightarrow+\infty}|t|^{2}\exp(-32\pi^{2}|t|^{2})\tilde{g}%
(t)<\infty,\text{ and }\lim_{t\rightarrow0}|t|^{-2}\tilde{g}(t)<\infty.
\]
By the Adams' inequality (\ref{int1}) with the exact growth in $\mathbb{R}^{4}$, we derive that%
\[
\int_{\mathbb{R}^{4}}\tilde{g}(u)dx\leq c\int_{\mathbb{R}^{4}}\frac{\Phi
(32\pi^{2}|u|^{2})}{(1+|u|)^{2}}dx\leq c\int_{\mathbb{R}^{4}}|u|^{2}dx
\]
Let~$v=(32\pi^{2}K)^{-1/2}u$. Then for any $u\in H^{2}(\mathbb{R}^{4})$
satisfying $\Vert\Delta u\Vert_{2}^{2}\leq32\pi^{2}K$, there holds%
\[
\int_{\mathbb{R}^{4}}g(u)dx=\int_{\mathbb{R}^{4}}\tilde{g}(v)dx\leq
c\int_{\mathbb{R}^{4}}|v|^{2}dx\leq c\int_{\mathbb{R}^{4}}|u|^{2}dx.
\]

Now, we turn to prove the sufficiency of (3). Let $g:\mathbb{R}\rightarrow
\lbrack0,+\infty)$ be a continuous function satisfying
\begin{equation}
\lim\limits_{t\rightarrow+\infty}|t|^{2}\exp(-\frac{|t|^{2}}{K}%
)g(t)=0\label{s1}%
\end{equation}
and
\begin{equation}
\lim_{t\rightarrow0}|t|^{-2}g(t)=0.\label{s2}%
\end{equation}
For any radially symmetric sequence $\{u_{k}\}_{k}$ satisfying $\Vert\Delta
u\Vert_{2}^{2}\leq32\pi^{2}K$, and weakly converging to $u$, we will verify
that
\[
\lim_{k\rightarrow\infty}G(u_{k})-G(u)=\lim_{k\rightarrow\infty}%
\int_{\mathbb{R}^{4}}(g(u_{k})-g(u))dx=0.
\]

Note that $\{u_{k}\}_{k}$~is a radial sequence in $H^{2}(\mathbb{R}^{4})$,
then
\[%
\begin{split}
\left\vert u_{k}(r)\right\vert ^{2}  &  \leq\int_{r}^{+\infty}2|u_{k}%
(s)||u_{k}^{\prime}(s)|ds\\
&  \leq cr^{-3}\int_{r}^{+\infty}|u_{k}(s)s^{3/2}||u_{k}^{\prime}%
(s)s^{3/2}|ds\\
&  \leq cr^{-3}\left(  \int_{\mathbb{R}^{4}}|\nabla u_{k}|^{2}dx\right)
^{1/2}\left(  \int_{\mathbb{R}^{4}}|u_{k}|^{2}dx\right)  ^{1/2}.
\end{split}
\]
Hence $u_{k}(r)\rightarrow0$ as $r\rightarrow\infty$ uniformly with respect to
$k$. This together with (\ref{s1}) yields that for any $\varepsilon>0$, there
exists $R>0$ such that
\begin{equation}
\int_{\mathbb{R}^{4}\setminus B_{R}}g(u_{k})dx\leq\varepsilon\int
_{\mathbb{R}^{4}}|u_{k}|^{2}dx\leq c\varepsilon,\int_{\mathbb{R}^{4}\setminus
B_{R}}g(u)dx\leq\varepsilon. \label{s3}%
\end{equation}
On the other hand, through (\ref{s2}) we derive that for any $\varepsilon>0$,
there exists $L>0$ independent of $k$ such that
\[
\int_{|u_{k}|>L}g(u_{k})dx\leq c\varepsilon\int_{|u_{k}|>L}\frac{\exp
(-\frac{1}{K}|u_{k}|^{2})}{|u_{k}|^{2}}dx
\]
and
\[
\int_{|u|>L}g(u)dx\leq c\varepsilon\int_{|u|>L}\frac{\exp(-\frac{1}{K}%
|u|^{2})}{|u|^{2}}dx.
\]
In view of (\ref{int1}), we derive that
\begin{equation}
\int_{|u_{k}|>L}g(u_{k})dx\leq c\varepsilon\int_{\mathbb{R}^{4}}|u_{k}%
|^{2}dx\leq c\varepsilon,\ \ \int_{|u|>L}g(u)dx\leq\varepsilon. \label{s4}%
\end{equation}
Combining (\ref{s3}) and (\ref{s4}), one can get%
\begin{align*}
\lim_{k\rightarrow\infty}\left\vert G(u_{k})-G(u)\right\vert  &  \leq\left(
\int_{\mathbb{R}^{4}\setminus B_{R}}+\int_{B_{R}}\right)  \left\vert
g(u_{k})-g(u)\right\vert dx\\
&  \leq c\varepsilon+\lim_{k\rightarrow\infty}\left(  \int_{|u_{k}|>L}%
g(u_{k})dx+\int_{|u|>L}g(u)dx\right) \\
&\ \  +\lim_{k\rightarrow\infty}\left(  \int_{|u_{k}|\leq L,|x|\leq R}%
g(u_{k})dx-\int_{|u|\leq L,|x|\leq R}g(u)dx\right) \\
&  \leq c\varepsilon+\lim_{k\rightarrow\infty}\left(  \int_{|u_{k}|\leq
L,|x|\leq R}g(u_{k})dx-\int_{|u|\leq L,|x|\leq R}g(u)dx\right) \\
&  \leq c\varepsilon,
\end{align*}
where we have used the Lebesgue dominated convergence
theorem in the last step. Then the proof is finished. \medskip
\end{proof}

\section{Existence of non-trivial solutions\ for semilinear bi-harmonic
equations}

In this section, we consider the nontrivial solutions of semilinear bi-harmonic equation
(\ref{bi-harmonic1}). We will employ the compactness result obtained in Theorem \ref{thm1} and the principle of symmetric criticality to
prove that equation (\ref{bi-harmonic1}) has a nontrivial radial solution under the assumption that the nonlinearity $f(t)$ satisfies mild conditions
(i), (ii) and (\ref{exponential critical}).

The natural functional associated to a variational approach to problem
(\ref{bi-harmonic1}) is
\[
I_{\gamma}(u)=\frac{1}{2}\left(  \left\Vert \Delta u\right\Vert _{2}%
^{2}+\gamma\left\Vert u\right\Vert _{2}^{2}\right)  -\int_{\mathbb{R}^{4}%
}F(u)dx,\ \forall\ u\in H^{2}(\mathbb{R}^{4}).
\]
Obviously, $I_{\gamma}\in C^{1}(H^{2}(\mathbb{R}^{4}),\mathbb{R})$ with
\[
I_{\gamma}^{\prime}(u)v=\int_{\mathbb{R}^{4}}(\Delta u\Delta v+\gamma
uv)dx-\int_{\mathbb{R}^{4}}f(u)vdx,\ \ \forall u,v\in H^{2}(\mathbb{R}^{4}).
\]
Our goal is to prove the existence of non-trivial solutions of the equation
(\ref{bi-harmonic1}). According to the principle of symmetric criticality, we
only need to verify that $u$ is a critical point restricted to the space
$H_{r}^{2}(\mathbb{R}^{4})$. Motivated by the Pohozaev identity for equation
(\ref{bi-harmonic1}), we introduce the functional
\[
G_{\gamma}(u)=\gamma\left\Vert u\right\Vert _{2}^{2}-2\int_{\mathbb{R}^{4}%
}F(u)dx
\]
and the constrained minimization problem%

\begin{equation}%
\begin{split}
A_{\gamma}  &  =\inf\left\{  \left.  \frac{1}{2}\Vert\Delta u\Vert_{2}%
^{2}\ \right\vert \ u\in H_{r}^{2}(\mathbb{R}^{4}),\ G_{\gamma}(u)=0\right\}
\\
&  =\inf\left\{  \left.  I_{\gamma}(u)\ \right\vert u\in H_{r}^{2}\left(
\mathbb{R}^{4}\right)  ,\ G_{\gamma}(u)=0\right\}  .
\end{split}
\label{constrained}%
\end{equation}

\vskip0.1cm

Set $\mathcal{P}_{r}=\left\{  u\in H_{r}^{2}(\mathbb{R}^{4}),\ G_{\gamma
}(u)=0\right\}$. Apparently, $\mathcal{P}_{r}$ is not empty. In
fact, let $u_{0}\in H_{r}^{2}(\mathbb{R}^{4})$ be compactly supported and
define
\[
h(s):=G_{\gamma}(su_{0})=\gamma s^{2}\left\Vert u_{0}\right\Vert _{2}%
^{2}-2\int_{\mathbb{R}^{4}}F(su_{0})dx,\forall s>0.
\]
It follows from the fact
\[
\lim_{s\rightarrow0^{+}}\frac{F(s)}{s^{2}}=0,\ \lim_{s\rightarrow+\infty}%
\frac{F(s)}{s^{2}}=\infty
\]
that $h(s)>0$ for $s>0$ small enough and $h(s)<0$ for $s>0$ sufficiently
large. Therefore, there exists $s_{0}>0$ such that $h(s_{0}u_{0})=0$. This gives
$s_{0}u_{0}\in\mathcal{P}_{r}$.

\begin{lemma}
\label{lemm1} There exists a minimizing sequence $\{u_{k}\}_{k}\in
\mathcal{P}_{r}$ satisfying $\left\Vert u_{k}\right\Vert _{2}=1$ for
$A_{\gamma}$.
\end{lemma}

\begin{proof}
Assume that $\{u_{k}\}_{k}$ is a minimizing sequence for $A_{\gamma}$, that
is, $u_{k}\in\mathcal{P}_{r}$ satisfying
\[
\lim\limits_{k\rightarrow\infty}\frac{1}{2}\left\Vert \Delta u_{k}\right\Vert
_{2}^{2}=A_{\gamma}.
\]
Let $\tilde{v}_{k}=u_{k}(\left\Vert u_{k}\right\Vert _{2}^{1/2}x)$, simple
computations lead to $\left\Vert \tilde{v}_{k}\right\Vert _{2}=1$,
$\tilde{v}_{k}\in\mathcal{P}_{r}$ and $\left\Vert \Delta v_{k}\right\Vert
_{2}=\left\Vert \Delta\tilde{v}_{k}\right\Vert _{2}$. This accomplishes the
proof of Lemma \ref{lemm1}.
\end{proof}
\vskip0.1cm

If the infimum $A_{\gamma}$ is attained, then the minimizer $u\in H_{r}%
^{2}(\mathbb{R}^{4})$ under a suitable change of scale is a ground state
solution of (\ref{bi-harmonic1}) constrained to the space $H_{r}%
^{2}(\mathbb{R}^{4})$. In fact, if $u$ is a minimizer for $A_{\gamma}$, then
there exists a Lagrange multiplier $\theta\in\mathbb{R}$ such that
\[
\Delta^{2}u+\gamma u-f(u)=\theta(2\gamma u-2f(u))\ \text{in}\ \mathbb{R}^{4}%
\]
namely,
\[
\Delta^{2}u=(2\theta-1)(\gamma u-f(u))\ \text{in}\ \mathbb{R}^{4}.
\]
Recalling that $u\in\mathcal{P}_{r}$, we have
\[%
\begin{split}
\int_{\mathbb{R}^{4}}(\gamma u-f(u))udx  &  =\gamma\left\Vert u\right\Vert
_{2}^{2}-2\int_{\mathbb{R}^{4}}F(u)dx+\int_{\mathbb{R}^{4}}(2F(u)-uf(u))dx\\
&  =-\int_{\mathbb{R}^{4}}\left(  uf(u)-2F(u)\right)  dx<0,
\end{split}
\]
as a consequence of (i). Moreover,
\[
\int_{\mathbb{R}^{4}}\Delta^{2}u\cdot udx=\int_{\mathbb{R}^{4}}|\Delta
u|^{2}dx>0,
\]
hence $2\theta-1<0$. Therefore
\begin{equation}
\tilde{u}(x)=u\left(  \frac{x}{(1-2\theta)^{\frac{1}{2}}}\right)
\ \text{for}\ \text{a.e.}\ x\in\mathbb{R}^{4} \label{rescal}%
\end{equation}
is a non-trivial solution of (\ref{bi-harmonic1}) constrained to the space
$H_{r}^{2}(\mathbb{R}^{4})$. According to the principle of symmetric
criticality, then $u$ is a non-trivial solution of (\ref{bi-harmonic1}).
\medskip

Now, we establish an relation between the attainability of $A_{\gamma}$ and
the Adams' inequality with the exact growth (\ref{int1}). For this purpose, we
introduce the Adams ratio
\[
C_{A}^{L}=\sup\{\frac{2}{\left\Vert u\right\Vert _{2}^{2}}\int_{\mathbb{R}%
^{4}}F(u)|\ u\in H_{r}^{2}(\mathbb{R}^{4}),\left\Vert \Delta u\right\Vert
_{2}^{2}\leq L\}.
\]
The Adams threshold $R(F)$ is given by
\[
R(F)=\sup\{L>0\ |\ C_{A}^{L}<+\infty\}.
\]
We denote by $C_{A}^{\ast}=C_{A}^{R(F)}$ the ratio at the threshold\ $R(F)$.
By the growth condition (\ref{exponential critical})\ and (ii) of $f\left(
s\right)  $, we obtain
\[
\underset{t\rightarrow+\infty}{\lim}\frac{t^{2}F\left(  t\right)  }%
{\exp\left(  \alpha t^{2}\right)  }=%
\genfrac{\{}{.}{0pt}{}{0\text{, if }\alpha>\alpha_{0},}{+\infty\text{, if
}\alpha<\alpha_{0},}%
\]
and
\[
\underset{t\rightarrow0^{+}}{\lim}\frac{F\left(  t\right)  }{t^{2}}=0\text{.}%
\]
Hence, thanks to Theorem \ref{thm1}, we derive $R(F)=32\pi^{2}/\alpha_{0}$.

\begin{lemma}
\label{lemm2} If $A_{\gamma}<R(F)/2$, then $A_{\gamma}$ can be attained and
$A_{\gamma}=I_{\gamma}(u)$, where $u\in H_{r}^{2}(\mathbb{R}^{4})$ under a
suitable change of scale is a nontrivial solution of equation (\ref{bi-harmonic1}%
) through the principle of symmetric criticality.
\end{lemma}

\begin{proof}
Let $\{u_{k}\}_{k}$ be a radial minimizing sequence for $A_{\gamma}$, that is
$u_{k}\in\mathcal{P}_{r}$ satisfying
\[
\lim\limits_{k\rightarrow\infty}\frac{1}{2}\left\Vert \Delta u_{k}\right\Vert
_{2}^{2}=A_{\gamma}\ \ \text{and}\ \ \left\Vert u_{k}\right\Vert _{2}^{2}=1.
\]
We also assume that $u_{k}\rightharpoonup u$ in $H^{2}(\mathbb{R}^{4})$. We
first prove that $A_{\gamma}>0$. We argue this by contradiction. We assume that
$A_{\gamma}=0$, nemely $\lim\limits_{k\rightarrow\infty}\left\Vert \Delta
u_{k}\right\Vert _{2}^{2}=0$, which implies that $u=0$. Regarding
\[
\ \lim\limits_{t\rightarrow+\infty}|t|^{2}\exp(-\alpha|t|^{2}%
)F(t)=0\ \text{for\ any }\alpha>\alpha_{0},\ \ \lim\limits_{t\rightarrow0}%
|t|^{-2}F(t)=0,
\]
we derive that
\[
\lim\limits_{k\rightarrow\infty}\int_{\mathbb{R}^{4}}F(u_{k})dx=\int
_{\mathbb{R}^{4}}F(u)dx
\]
through Theorem \ref{thm2}.
On the other hand, since $u_{k}\in\mathcal{P}_{r}$ and $\left\Vert
u_{k}\right\Vert _{2}^{2}=1$, then
\[
0<\gamma \lim_{k\rightarrow\infty}\left\Vert u_{k}\right\Vert _{2}^{2}=2\lim\limits_{k\rightarrow\infty
}\int_{\mathbb{R}^{4}}F(u_{k})dx=2\int_{\mathbb{R}^{4}}F(u)dx,
\]
which contradicts $u=0$. This proves that $A_{\gamma}>0$. \medskip

Now are in position to prove that if $A_{\gamma}<R(F)/2$, then $A_{\gamma}$
could be attained. Under the assumption of Lemma \ref{lemm2}, we have
$$\lim_{k\rightarrow\infty}\left\Vert \Delta u_{k}\right\Vert _{2}%
^{2}=2A_{\gamma}<R(F)=32\pi^{2}/\alpha_{0}.$$ Picking up $\frac{1}{K}%
>\alpha_{0}$ satisfying $\lim_{k\rightarrow\infty}\left\Vert \Delta
u_{k}\right\Vert _{2}^{2}\leq32\pi^{2}K$, then we derive that
\[
\ \lim\limits_{t\rightarrow+\infty}|t|^{2}\exp(-\frac{1}{K}|t|^{2}%
)F(t)=0,\ \lim\limits_{t\rightarrow0}|t|^{-2}F(t)=0.
\]
It follows from Theorem \ref{thm2} that
\begin{equation}\label{ad1}
\lim\limits_{k\rightarrow\infty}\int_{\mathbb{R}^{4}}F(u_{k})dx=\int
_{\mathbb{R}^{4}}F(u)dx.
\end{equation}
Consequently,
\[
\gamma=\lim\limits_{k\rightarrow\infty}\gamma\left\Vert u_{k}\right\Vert
_{2}^{2}=2\lim\limits_{k\rightarrow\infty}\int_{\mathbb{R}^{4}}F(u_{k}%
)dx=2\int_{\mathbb{R}^{4}}F(u)dx
\]
and
\[
\frac{1}{2}\left\Vert \Delta u\right\Vert _{2}^{2}\leq\lim
\limits_{k\rightarrow\infty}\frac{1}{2}\left\Vert \Delta u_{k}\right\Vert
_{2}^{2}=A_{\gamma}.
\]
In order to show $u$ is minimizer for $A_{\gamma}$, what left is to show that
$G_{\gamma}(u)=0$. Set
\[
h(t)=G_{\gamma}(tu)=\gamma\left\Vert tu\right\Vert _{2}^{2}-\int
_{\mathbb{R}^{4}}F(tu)dx.
\]
Obviously, in view of \eqref{ad1}, we have
\begin{align*}
G_{\gamma}(u)  &  =\gamma\left\Vert u\right\Vert _{2}^{2}-2\int_{\mathbb{R}%
^{4}}F(u)dx\\
&  \leq\lim\limits_{k\rightarrow\infty}\left(  \gamma\left\Vert u_{k}%
\right\Vert _{2}^{2}-\int_{\mathbb{R}^{4}}F(u_{k})dx\right)  =\lim
\limits_{k\rightarrow\infty}G_{\gamma}(u_{k})=0.
\end{align*}
This implies  $h(1)\leq0$. From $\lim_{t\rightarrow0^{+}}\frac{F(t)}{t^{2}}=0$,
one can deduce that $h(t)>0$ for $t>0$ small enough. Consequently, there exists
$s_{0}\in(0,1]$ such that $G_{\gamma}(s_{0}u)=0$. Then it follows that
\[
A_{\gamma}\leq\frac{1}{2}\left\Vert \Delta s_{0}u\right\Vert _{2}^{2}=\frac
{1}{2}s_{0}^{2}\left\Vert \Delta u\right\Vert _{2}^{2}\leq s_{0}^{2}A_{\gamma
},
\]
which proves that $s_{0}=1$ and $\frac{1}{2}\left\Vert \Delta u\right\Vert
_{2}^{2}=A_{\gamma}$. Then we accomplish the proof of Lemma \ref{lemm2}.
\end{proof}

Next, we show

\begin{lemma}
\label{lemm5} The constrained minimization problem $A_{\gamma}$ associated to
the functional $I_{\gamma}$ satisfies
\[
A_{\gamma}<\frac{1}{2}R(F)
\]
if and only if
\[
\gamma<C_{A}^{\ast}.
\]

\end{lemma}

\begin{proof}
We first prove that if $A_{\gamma}<R(F)/2$, then $\gamma<C_{A}^{\ast}$.
Obviously, if the $C_{A}^{\ast}=+\infty$, then $\gamma<C_{A}^{\ast}$ and the
proof is complete. Therefore, without loss of generality, we may assume that
$C_{A}^{\ast}<+\infty$. According to Lemma \ref{lemm2}, we see that
$A_{\gamma}$ could be achieved by a radial function $u\in\mathcal{P}_{r}$.
Then according to the definition of the $A_{\gamma}$, we have $\left\Vert
\Delta u\right\Vert _{2}^{2}<32\pi^{2}/\alpha_{0}$ and $\gamma\left\Vert
u\right\Vert _{2}^{2}=2\int_{\mathbb{R}^{4}}F(u)dx$. Define
\[
g(s)=\frac{2}{s^{2}\left\Vert u\right\Vert _{2}^{2}}\int_{\mathbb{R}^{4}%
}F(su)dx,
\]
then $g(1)=\gamma$. Since $F$ satisfies the condition (i), then it is easy to
see that $g(s)$ is monotone increasing. If we set $v=\frac{R(F)^{1/2}%
}{\left\Vert \Delta u\right\Vert _{2}}u$, then $\Vert\Delta v\Vert_{2}%
^{2}=R(F)$ and \
\[
C_{A}^{\ast}\geq\frac{2}{\Vert v\Vert_{2}^{2}}\int_{\mathbb{R}^{4}%
}F(v)dx=g(\frac{R(F)^{1/2}}{\Vert\Delta u\Vert_{2}})>g(1)=\gamma.
\]

Next, it remains to verify  that if $\gamma<C_{A}^{\ast}$, then $A_{\gamma
}<R(F)/2$. We distinguish between the case $C_{A}^{\ast}<+\infty$ and
$C_{A}^{\ast}=+\infty$.
\vskip0.1cm

In the case $C_{A}^{\ast}<+\infty$, since
$\gamma<C_{A}^{\ast}$, then $\gamma<C_{A}^{\ast}-\varepsilon_{0}$ for some
$\varepsilon_{0}>0$. It follows from the definition of  $C_{A}^{\ast}$ that
 there exists some $u_{0}\in H_{r}^{2}(\mathbb{R}^{4})$ with $\Vert\Delta
u_{0}\Vert_{2}^{2}\leq R(F)$ satisfying
\[
C_{A}^{\ast}-\varepsilon_{0}<\frac{2}{\Vert u_{0}\Vert_{2}^{2}}\int
_{\mathbb{R}^{4}}F(u_{0})dx.
\]
Consequently,
\[
\gamma\Vert u_{0}\Vert_{2}^{2}<2\int_{\mathbb{R}^{4}}F(u_{0})dx,
\]
namely $G_{\gamma}(u)<0$. Let $h(s)=G_{\gamma}(su_{0})$ for $s>0$. Since
$h(1)<0$ and $h(s)>0$ for $s>0$ small enough, then there exists $s_{0}%
\in(0,1)$ satisfying $h(s_{0}u_{0})=0$. Therefore, we have  $s_{0}u_{0}\in
\mathcal{P}_{r}$ and
\[
A_{\gamma}\leq\frac{1}{2}\Vert\Delta(s_{0}u_{0})\Vert_{2}^{2}=\frac{1}{2}%
s_{0}^{2}\Vert\Delta u_{0}\Vert_{2}^{2}<\frac{1}{2}R(F).
\]
\vskip0.1cm

In the case $C_{A}^{\ast}=+\infty$, for any $\gamma>0$, there exists $u_{0}\in
H_{r}^{2}(\mathbb{R}^{4})$ with $\Vert\Delta u_{0}\Vert_{2}^{2}\leq R(F)$
satisfying
\[
\gamma\Vert u_{0}\Vert_{2}^{2}<2\int_{\mathbb{R}^{4}}F(u_{0})dx.
\]
Hence we can repeat the same arguments as  case $C_{A}^{\ast}<+\infty$ to get the conclusion.
\end{proof}

\section{Existence of ground state solutions for bi-harmonic equation with
the constant potential}

In this section, we will employ the Pohozaev manifold and Fourier rearrangement arguments to study the ground-states of the following semilinear bi-harmonic equation.
\begin{equation}
(-\Delta)^{2}u+\gamma u=\lambda u\exp(2|u|^{2})\ \text{in}~\mathbb{R}%
^{4},\ \ \label{7.1}%
\end{equation}
where $\lambda$ is strictly smaller than the first eigenvalue of operator
$\left(  -\Delta\right)  ^{2}+\gamma I$ in $\mathbb{R}^{4}$, namely
\[
\lambda<\inf_{u\in H^{2}(\mathbb{R}^{4})}\frac{\Vert\Delta u\Vert_{2}%
^{2}+\gamma\Vert u\Vert_{2}^{2}}{\Vert u\Vert_{2}^{2}}=\gamma.
\]

The natural functional associated to a variational approach to problem
(\ref{7.1}) is
\[
I_{\lambda}(u)=\frac{1}{2}\left(  \Vert\Delta u\Vert_{2}^{2}+\gamma\Vert
u\Vert_{2}^{2}\right)  -\frac{\lambda}{4}\int_{\mathbb{R}^{4}}\left(
\exp(2u^{2})-1\right)  dx,\ \forall u\in H^{2}(\mathbb{R}^{4}).
\]
It is easy to obtain that $I_{\lambda}\in C^{1}(H^{2}(\mathbb{R}^{4}),\mathbb{R})$ with
\[
I_{\lambda}^{\prime}(u)v=\int_{\mathbb{R}^{4}}(\Delta u\Delta v+\gamma
uv)dx-\int_{\mathbb{R}^{4}}\lambda u\exp(2u^{2})vdx,\ \forall u,v\in
H^{2}(\mathbb{R}^{4}).
\]
We will prove that equation (\ref{7.1}) has
a radial ground-state solution for any $0<\lambda<\gamma$.
\vskip0.1cm

We recall that a solution $u$ of (\ref{7.1})
is called a ground state if $I_{\lambda}(u)=m_{\lambda}$, where
\[
m_{\lambda}=\inf\{I_{\lambda}(u)\ |\ u\ \mathrm{is\ a\ weak\ solution\ of}%
\ \text{(\ref{7.1})}\}.
\]
Similar to the proof of Theorem \ref{thm3}, we introduce the Pohozaev
functional
\[
G_{\lambda}(u)=c\Vert u\Vert_{2}^{2}-\frac{1}{2}\int_{\mathbb{R}^{4}}%
\lambda\left(  \exp(2|u|^{2})-1\right)  )dx=(\gamma-\lambda)\Vert u\Vert
_{2}^{2}-\int_{\mathbb{R}^{4}}g_{\lambda}(u)dx,
\]
where $g_{\lambda}(t)=\frac{\lambda}{2}\left(  \exp(2t^{2})-1-2t^{2}\right)$, and the constrained minimization problem
\[%
\begin{split}
A_{\lambda} &  =\inf\left\{  \frac{1}{2}\Vert\Delta u\Vert_{2}^{2}\ |\ u\in
H^{2}(\mathbb{R}^{4}),\ G_{\lambda}(u)=0\right\}  \\
&  =\inf\left\{  I_{\lambda}(u)\ |\ u\in H^{2}(\mathbb{R}^{4}),\ G_{\lambda
}(u)=0\right\}  \leq m_{\lambda},
\end{split}
\]
 Set $\mathcal{P}=\left\{  \ u\in H^{2}(\mathbb{R}^{4}),\ G_{\lambda
}(u)=0\right\}  $, obviously $\mathcal{P}$ is not empty. Next, we will adapt the Fourier rearrangement method to show that there exists a radially minimizing sequence for $A_{\lambda}$.
Such a Fourier rearrangement argument has also been used recently by Chen, Lu and Zhang in \cite{ChenLuZhang}
to establish the existence of extremals for the subcritical Adams inequalities on the entire space.

\begin{lemma}
\label{lem1} There exists a radially minimizing sequence $\{u_{k}\}_{k}$
satisfying $\|u_{k}\|_{2}^{2}=1$ for $A_{\lambda}$.
\end{lemma}

\begin{proof}
Assume that $\{u_{k}\}_{k}$ is a minimizing sequence for $A_{\lambda}$, that
is $u_{k}\in\mathcal{P}$ satisfying
\[
\lim\limits_{k\rightarrow\infty}\frac{1}{2}\Vert\Delta u_{k}\Vert_{2}%
^{2}=A_{\lambda}.
\]
Denote by $w_{k}=\mathcal{F}^{-1}\{(\mathcal{F}(u_{k}))^{\ast}\}$ the Fourier
rearrangement of $u_{k}$, where $\mathcal{F}$ is the Fourier transform on
$\mathbb{R}^{4}$ (with its inverse $\mathcal{F}^{-1}$) and $f^{\ast}$ stands
for the Schwarz symmetrization of $f$. Using the property of the Fourier
rearrangement from \cite{Lenzmann}, one can derive that
\begin{align*}
\Vert\Delta w_{k}\Vert_{2}  &  \leq\Vert\Delta u_{k}\Vert_{2},\Vert w_{k}%
\Vert_{2}^{2}=\Vert u_{k}\Vert_{2}^{2},\ \\
\ \int_{\mathbb{R}^{4}}\left(  \exp(2w_{k}^{2})-1\right)  dx  &  \geq
\int_{\mathbb{R}^{4}}\left(  \exp(2u_{k}^{2})-1\right)  dx.
\end{align*}
Then it follows that
\[
(\gamma-\lambda)\Vert w_{k}\Vert_{2}^{2}=(\gamma-\lambda)\Vert u_{k}\Vert
^2_{2}=\int_{\mathbb{R}^{4}}g_{\lambda}(u_{k})dx\leq\int_{\mathbb{R}^{4}%
}g_{\lambda}(w_{k}).
\]
Hence if we set
\[
\eta(t)=(\gamma-\lambda)\Vert tw_{k}\Vert_{2}^{2}-\int_{\mathbb{R}^{4}%
}g_{\lambda}(tw_{k}),
\]
then $\eta(1)\leq0$. On the other hand, one can easily see $\eta(t)>0$ for $t>0$ sufficiently small.
Therefore, there exists $t_{k}\in(0,1]$ such that $\eta(t_{k})=0$, that is
$t_{k}w_{k}\in\mathcal{P}$. We obtain
\[
m_{\lambda}\leq I_{\lambda}(t_{k}w_{k})=\frac{1}{2}\Vert\Delta(t_{k}%
w_{k})\Vert_{2}^{2}\leq\frac{1}{2}t_{k}^{2}\Vert\Delta u_{k}\Vert_{2}^{2}\leq
I_{\lambda}(u_{k}).
\]
This implies that $\{v_{k}\}:=\{t_{k}w_{k}\}_{k}$ is a radial minimizing
sequence for $m_{\lambda}$. Let $\tilde{v}_{k}=w_{k}(\Vert v_{k}\Vert
_{2}^{1/2}x)$, it is easy to check that $\tilde{v}_{k}$ is a minimizing sequence for $A_{\lambda}$ with $\Vert\tilde{v}_{k}\Vert_{2}=1$. This accomplishes the proof of Lemma \ref{lem1}.
\end{proof}
\vskip0.1cm

Repeating the argument for (\ref{rescal}), we can show that if the infimum
$A_{\lambda}$ is attained, then the minimizer $u\in H_{r}^{2}(\mathbb{R}^{4})$
 under a suitable change of scale is a ground state solution of (\ref{7.1}).

\begin{lemma}
\label{lem2} If $A_{\lambda}<8\pi^{2}$, then $A_{\lambda}$ could be attained
and $A_{\lambda}=I_{\lambda}(u)$, where $u\in H_{r}^{2}(\mathbb{R}^{4})$
under a suitable change of scale is a ground-state solution of equation
(\ref{7.1}).
\end{lemma}

\begin{proof}
Let $\{u_{k}\}_{k}$ is a radial minimizing sequence for $A_{\lambda}$, that is
$u_{k}\in\mathcal{P}$ satisfying
\[
\lim\limits_{k\rightarrow\infty}\frac{1}{2}\Vert\Delta u_{k}\Vert_{2}%
^{2}=A_{\lambda}\ \ \text{and}\ \ \Vert u_{k}\Vert_{2}^{2}=1.
\]
We also assume that $u_{k}\rightharpoonup u$ in $H^{2}(\mathbb{R}^{4})$. We
first prove that $A_{\lambda}>0$. By way of contradiction, we assume that
$A_{\lambda}=0$, namely $\lim\limits_{k\rightarrow\infty}\Vert\Delta
u_{k}\Vert_{2}^{2}=0$. This implies $u=0$. Since
\[
\ \lim\limits_{t\rightarrow+\infty}|t|^{2}\exp(-\alpha|t|^{2})g_{\lambda
}(t)=0\ \text{for\ any }\alpha>2,\ \ \lim\limits_{t\rightarrow0}|t|^{-2}%
g_{\lambda}(t)=0.
\]
It follows from Theorem \ref{thm2} that
\[
\lim\limits_{k\rightarrow\infty}\int_{\mathbb{R}^{4}}g_{\lambda}(u_{k}%
)dx=\int_{\mathbb{R}^{4}}g_{\lambda}(u)dx.
\]
On the other hand, since $u_{k}\in\mathcal{P}$ with $\Vert u_{k}\Vert_{2}%
^{2}=1$, we have
\[
0<(\gamma-\lambda)\leq\lim\limits_{k\rightarrow\infty}(\gamma-\lambda)\Vert
u_{k}\Vert_{2}^{2}=\lim\limits_{k\rightarrow\infty}\int_{\mathbb{R}^{4}%
}g_{\lambda}(u_{k})dx=\int_{\mathbb{R}^{4}}g_{\lambda}(u)dx,
\]
which contradicts $u=0$. This proves that $A_{\lambda}>0$. \medskip

Now we are in position to prove that if $A_{\lambda}<8\pi^{2}$, then
$A_{\lambda}$ could be attained. Under the assumption of Lemma \ref{lem2}, we
derive that $\lim_{k\rightarrow\infty}\Vert\Delta u_{k}\Vert_{2}^{2}=2A_{\lambda
}<16\pi^{2}$. Setting $K=A_{\lambda}/16\pi^{2}$, observing that
\[
\ \lim\limits_{t\rightarrow+\infty}|t|^{2}\exp(-\frac{1}{K}|t|^{2})g_{\lambda
}(t)=0\text{ and}\ \lim\limits_{t\rightarrow0}|t|^{-2}g_{\lambda}(t)=0,
\]
one can employ the compactness result obtained in Theorem \ref{thm2} to derive that
\[
\lim_{k\rightarrow\infty}\int_{\mathbb{R}^{4}}g_{\lambda}(u_{k})dx=\int
_{\mathbb{R}^{4}}g_{\lambda}(u)dx.
\]
Consequently, we get
\[
(\gamma-\lambda)=\lim\limits_{k\rightarrow\infty}(\gamma-\lambda)\Vert
u_{k}\Vert_{2}^{2}=\lim\limits_{k\rightarrow\infty}\int_{\mathbb{R}^{4}%
}g_{\lambda}(u_{k})dx=\int_{\mathbb{R}^{4}}g_{\lambda}(u)dx
\]
and
\begin{equation}
\frac{1}{2}\Vert\Delta u\Vert_{2}^{2}\leq\lim\limits_{k\rightarrow\infty}%
\frac{1}{2}\Vert\Delta u_{k}\Vert_{2}^{2}=A_{\lambda}. \label{lim}%
\end{equation}
In order to show $u$ is minimizer for $A_{\lambda}$, we only need to verify that
$G_{\lambda}(u)=0$. Set
\[
h(t)=G_{\lambda}(tu)=(\gamma-\lambda)\Vert tu\Vert_{2}^{2}-\int_{\mathbb{R}%
^{4}}g_{\lambda}(tu)dx.
\]
Since
\begin{align*}
G_{\lambda}(u)  &  =(\gamma-\lambda)\Vert u\Vert_{2}^{2}-\int_{\mathbb{R}^{4}%
}g_{\lambda}(u)dx\\
&  \leq\lim\limits_{k\rightarrow\infty}(\gamma-\lambda)\Vert u_{k}\Vert
_{2}^{2}-\int_{\mathbb{R}^{4}}g_{\lambda}(u_{k})dx=\lim\limits_{k\rightarrow
\infty}G_{\lambda}(u_{k})=0,
\end{align*}
then $h(1)\leq0$.  In view of $\lim_{t\rightarrow0^{+}}\frac{g_{\lambda}(t)}{t^{2}}=0$,
one can deduce that $h(t)>0$ for $t>0$ small enough. Consequently, there exists
$s_{0}\in(0,1]$ such that $G_{\lambda}(s_{0}u)=0$. By (\ref{lim}),
we obtain
\[
A_{\lambda}\leq\frac{1}{2}\Vert\Delta s_{0}u\Vert_{2}^{2}=\frac{1}{2}s_{0}%
^{2}\Vert\Delta u\Vert_{2}^{2}\leq s_{0}^{2}A_{\lambda},
\]
which leads to $s_{0}=1$ and $\frac{1}{2}\Vert\Delta u\Vert_{2}%
^{2}=A_{\lambda}$. Then we accomplish the proof of Lemma \ref{lem2}.
\end{proof}

Next, we prove the

\begin{lemma}
\label{lem5} The constrained minimization problem $A_{\lambda}$ is actually
strictly smaller than $8\pi^{2}$.
\end{lemma}

\begin{proof}
Note that the Adams ratio for $g_{\lambda}(u)$ is $+\infty$, hence there exists $u_{0}\in
H^{2}(\mathbb{R}^{4})$ such that
\[
(\gamma-\lambda)\leq\frac{1}{\Vert u_{0}\Vert_{2}^{2}}\int_{\mathbb{R}^{4}%
}g_{\lambda}(u_{0})dx,\ \ \Vert\Delta u_{0}\Vert_{2}^{2}\leq16\pi^{2},
\]
thus, we have $G_{\lambda}(u_{0})=(\gamma-\lambda)\Vert u_{0}\Vert_{2}^{2}%
-\int_{\mathbb{R}^{4}}g_{\lambda}(u_{0})dx<0$. Then there exists $s_{0}%
\in(0,1)$ such that $s_{0}u_{0}\in\mathcal{P}$, which yields that
\[
A_{\lambda}\leq\frac{1}{2}\Vert\Delta(s_{0}u_{0})\Vert_{2}^{2}=\frac{1}%
{2}s_{0}^{2}\Vert\Delta u_{0}\Vert_{2}^{2}\leq8\pi^{2}s_{0}^{2}<8\pi^{2}.
\]
Then the lemma is proved.
\end{proof}

\section{Existence of ground state solutions for bi-harmonic equation with
the Rabinowitz type potential}

In this section, we are concerned with the ground states of the following quasilinear bi-harmonic equation with the Rabinowitz type potential%

\begin{equation}
\left(  -\Delta\right)  ^{2}u+V\left(  x\right)  u=\lambda\exp\left(
2u^{2}\right)  u \label{variable},%
\end{equation}
where $\lambda$ and $V\left(  x\right)  $ satisfy
\[
0<\lambda<V_{0}=\underset{x\in\mathbb{R}^{4}}{\inf}V\left(  x\right)
<\underset{\left\vert x\right\vert \rightarrow\infty}{\lim}V\left(  x\right)
=\gamma.
\]

The associated functional and Nehari Manifold are
\[
I_{V}\left(  u\right)  =\frac{1}{2}\int_{\mathbb{R}^{4}}\left(  \left\vert
\Delta u\right\vert ^{2}+V\left(  x\right)  \left\vert u\right\vert
^{2}\right)  dx-\frac{\lambda}{4}\int_{\mathbb{R}^{4}}\left(  \exp\left(
2u^{2}\right)  -1\right)  dx
\]
and%

\[
\mathcal{N}_{V}=\left\{  \left.  u\in H^{2}\left(  \mathbb{R}^{4}\right)
\right\vert u\neq0,N_{V}\left(  u\right)  =0\right\}  ,
\]
respectively, where
\[
N_{V}\left(  u\right)  =\int_{\mathbb{R}^{4}}\left(  \left\vert \Delta
u\right\vert ^{2}+V\left(  x\right)  \left\vert u\right\vert ^{2}\right)
dx-\lambda\int_{\mathbb{R}^{4}}\exp\left(  2u^{2}\right)  u^{2}dx.
\]

In order to study the equation (\ref{variable}), we introduce the following
limiting equation
\begin{equation}
\left(  -\Delta\right)  ^{2}u+\gamma u=\lambda\exp\left(  2u^{2}\right)  u.
\label{limit}%
\end{equation}
The corresponding functional and Nehari Manifold associated with (\ref{limit}) is%

\[
I_{\infty}\left(  u\right)  =\frac{1}{2}\int_{\mathbb{R}^{4}}\left(
\left\vert \Delta u\right\vert ^{2}+\gamma\left\vert u\right\vert ^{2}\right)
dx-\frac{\lambda}{4}\int_{\mathbb{R}^{4}}\left(  \exp\left(  2u^{2}\right)
-1\right)  dx
\]
and%

\[
\mathcal{N}_{\infty}=\left\{  \left.  u\in H^{2}\left(  \mathbb{R}^{4}\right)
\right\vert u\neq0,N_{\infty}\left(  u\right)  =0\right\}  ,
\]
where
\[
N_{\infty}\left(  u\right)  =\int_{\mathbb{R}^{4}}\left(  \left\vert \Delta
u\right\vert ^{2}+\gamma\left\vert u\right\vert ^{2}\right)  dx-\lambda
\int_{\mathbb{R}^{4}}\exp\left(  2u^{2}\right)  u^{2}dx.
\]

One can easily verify that if $u\in\mathcal{N}_{V}$, then
\[
I_{V}\left(  u\right)  =\frac{\lambda}{4}\int_{\mathbb{R}^{4}}\left(
\exp\left(  2u^{2}\right)  2u^{2}-\left(  \exp\left(  2u^{2}\right)
-1\right)  \right)  dx,
\]
and if $u\in\mathcal{N}_{\infty}$, then
\[
I_{\infty}\left(  u\right)  =\frac{\lambda}{4}\int_{\mathbb{R}^{4}}\left(
\exp\left(  2u^{2}\right)  2u^{2}-\left(  \exp\left(  2u^{2}\right)
-1\right)  \right)  dx.
\]

\begin{lemma}
\bigskip\ For any $u\in H^{2}\left(  \mathbb{R}^{4}\right)  $, there exist
unique $t_{u}$ and $\tilde{t}_{u}$ such that $t_{u}u\in\mathcal{N}_{q}$ and
$\tilde{t}_{u}u\in\mathcal{N}_{\infty}$.
\end{lemma}

\begin{proof}
For any $u\in H^{2}\left(  \mathbb{R}^{4}\right)  $, we have
\[
N_{V}\left(  tu\right)  =t^{2}\int_{\mathbb{R}^{4}}\left(  \left\vert \Delta
u\right\vert ^{2}+\left(  V\left(  x\right)  -\lambda\right)  \left\vert
u\right\vert ^{2}\right)  dx-\lambda\int_{\mathbb{R}^{4}}\left(  \exp\left(
2t^{2}u^{2}\right)  -1\right)  t^{2}u^{2}dx.
\]
Since
\[
\underset{t\rightarrow0}{\lim}\frac{\left(  \exp\left(  2t^{2}u^{2}\right)
-1\right)  t^{2}u^{2}}{t^{2}}=0\text{ and}\underset{t\rightarrow\infty}{\lim
}\frac{\left(  \exp\left(  2t^{2}u^{2}\right)  -1\right)  t^{2}u^{2}}{t^{2}%
}=+\infty,
\]
then $N_{V}\left(  tu\right)  >0$ for $s>0$ small enough and $N_{V}\left(
tu\right)  <0$ for $t$ sufficient large. With the help of the monotonicity of
$\exp\left(  2t^{2}u^{2}\right)$, there exists a unique $t_{u}>0$ such that
$t_{u}u\in\mathcal{N}_{V}$. The proof for $\mathcal{N}_{\infty}$ is similar.
\end{proof}
\medskip

Set
\[
m_{\infty}=\inf\left\{  I_{\infty}\left(  u\right)  ,u\in\mathcal{N}%
_{\infty}\right\}  \text{ and }m_{V}=\inf\left\{  I_{V}\left(  u\right)
,u\in\mathcal{N}_{V}\right\}  .
\]
\bigskip From Corollary \ref{coroll}, we know that $m_{\infty}$ is attained by some
$w\in\mathcal{N}_{\infty}$.

\begin{lemma}
\label{impor}There holds
\begin{equation}
0<m_{V}<m_{\infty}. \label{the rang}%
\end{equation}

\end{lemma}

\begin{proof}
We first claim that for any $u\in m_{V}$, there exists some \emph{bounded}
$\tilde{t}_{u\text{ }}>1$ such that $\tilde{t}_{u}u\in\mathcal{N}_{\infty}$.
Indeed, for any $u\in\mathcal{N}_{V}$, by the assumption of $V\left(
x\right)  $, we have
\begin{align*}
\lambda\int_{\mathbb{R}^{4}}\exp\left(  2u^{2}\right)  u^{2}dx  &
=\int_{\mathbb{R}^{4}}\left(  \left\vert \Delta u\right\vert ^{2}+V\left(
x\right)  \left\vert u\right\vert ^{2}\right)  dx\\
&  <\int_{\mathbb{R}^{4}}\left(  \left\vert \Delta u\right\vert ^{2}%
+\gamma\left\vert u\right\vert ^{2}\right)  dx=\lambda\int_{\mathbb{R}^{4}%
}\exp\left(  2\tilde{t}_{u}^{2}u^{2}\right)  u^{2}dx,
\end{align*}
which implies $\tilde{t}_{u\text{ }}>1$. Now, we prove the boundedness of $\tilde
{t}_{u\text{ }}$. By the assumption of $V\left(  x\right)  $, there exists
some constant $c>0$ such that
\begin{align}
c\lambda\int_{\mathbb{R}^{4}}\exp\left(  2u^{2}\right)  u^{2}dx  &
=c\int_{\mathbb{R}^{4}}\left(  \left\vert \Delta u\right\vert ^{2}+V\left(
x\right)  \left\vert u\right\vert ^{2}\right)  dx\nonumber\\
&  \geq\int_{\mathbb{R}^{4}}\left(  \left\vert \Delta u\right\vert
^{2}+\left(  \gamma-\lambda\right)  \left\vert u\right\vert ^{2}\right)
dx\nonumber\\
&  =\lambda\int_{\mathbb{R}^{4}}\left(  \exp\left(  2\tilde{t}_{u\text{ }}%
^{2}u^{2}\right)  -1\right)  u^{2}dx. \label{a1}%
\end{align}
Since $\lambda\int_{\mathbb{R}^{4}}\left(  \exp\left(  2u^{2}\right)
-1\right)  u^{2}dx=\int_{\mathbb{R}^{4}}\left(  \left\vert \Delta u\right\vert
^{2}+\left(  V\left(  x\right)  -\lambda\right)  \left\vert u\right\vert
^{2}\right)  dx$, we derive that%
\begin{align}
\lambda\int_{\mathbb{R}^{4}}\left(  \exp\left(  2\tilde{t}_{u}^{2}%
u^{2}\right)  -1\right)  u^{2}dx  &  \geq\lambda\tilde{t}_{u\text{ }}^{2}%
\int_{\mathbb{R}^{4}}\left(  \exp\left(  2u^{2}\right)  -1\right)
u^{2}dx\nonumber\\
&  =\tilde{t}_{u}^{2}\int_{\mathbb{R}^{4}}\left(  \left\vert \Delta
u\right\vert ^{2}+\left(  V\left(  x\right)  -\lambda\right)  \left\vert
u\right\vert ^{2}\right)  dx. \label{a2}%
\end{align}
Combining (\ref{a1}) and (\ref{a2}), we conclude that
\begin{align*}
c\int_{\mathbb{R}^{4}}\left(  \left\vert \Delta u\right\vert ^{2}+V\left(
x\right)  \left\vert u\right\vert ^{2}\right)  dx  &  =c\lambda\int
_{\mathbb{R}^{4}}\exp\left(  2u^{2}\right)  u^{2}dx\\
&  \geq\lambda\int_{\mathbb{R}^{4}}\left(  \exp\left(  2\tilde{t}_{u\text{ }%
}^{2}u^{2}\right)  -1\right)  u^{2}dx\\
&  \geq\tilde{t}_{u\text{ }}^{2}\int_{\mathbb{R}^{4}}\left(  \left\vert \Delta
u\right\vert ^{2}+\left(  V\left(  x\right)  -\lambda\right)  \left\vert
u\right\vert ^{2}\right)  dx.
\end{align*}
Therefore, $\tilde{t}_{u\text{ }}$must be bounded.
\vskip0.1cm

To show that $m_{V}<m_{\infty}$, it is enough to build a sequence $\left(
u_{k}\right)  _{k}$ satisfying $u_{k}\in\mathcal{N}_{V}$ such that $\underset
{k\rightarrow\infty}{\lim}I_{V}\left(  u_{k}\right)  \leq m_{\infty}.$ Consider
$\left(  y_{k}\right)  _{k}$ with $y_{k}\in\mathbb{R}^{4}%
$ and $\left\vert y_{k}\right\vert \rightarrow\infty$.
We define $u_{k}=t_{k}w_{y_{k}}$, where $w_{y_{k}}=w\left(  \cdot-y_{k}\right)
$ and $t_{k}=t_{w_{y_{k}}}$ satisfying $u_{k}=t_{k}w_{y_{k}}\in$ $\mathcal{N}%
_{V}$. Then we have%
\[
\int_{\mathbb{R}^{4}}\left(  \left\vert \Delta w\right\vert ^{2}%
+\gamma\left\vert w\right\vert ^{2}\right)  dx=\lambda\int_{\mathbb{R}^{4}%
}\exp\left(  2w^{2}\right)  w^{2}dx
\]
and%
\begin{align*}
\lambda\int_{\mathbb{R}^{4}}\exp\left(  2t_{k}^{2}w^{2}\right)  w^{2}dx  &
=\int_{\mathbb{R}^{4}}\left(  \left\vert \Delta w_{y_{k}}\right\vert
^{2}+V\left(  x\right)  \left\vert w_{y_{k}}\right\vert ^{2}\right)  dx\\
&  =\int_{\mathbb{R}^{4}}\left(  \left\vert \Delta w\right\vert ^{2}+V\left(
x+y_{k}\right)  \left\vert w\right\vert ^{2}\right)  dx\\
&  \rightarrow\int_{\mathbb{R}^{4}}\left(  \left\vert \Delta w\right\vert
^{2}+\gamma\left\vert w\right\vert ^{2}\right)  dx\\
&  =\lambda\int_{\mathbb{R}^{4}}\exp\left(  2w^{2}\right)  w^{2}dx.
\end{align*}
This gives $t_{k}\rightarrow1$ as $k\rightarrow\infty$. Observe that%
\begin{equation}%
\begin{array}
[c]{l}%
c\left(  \exp\left(  2u^{2}\right)  2u^{2}-2u^{2}\right)  \leq\exp\left(
2u^{2}\right)  2u^{2}-\exp\left(  2u^{2}\right)  +1\\
=\frac{\left(  2u^{2}\right)  ^{2}}{2}+\left(  \frac{1}{2}-\frac{1}%
{3!}\right)  \left(  2u^{2}\right)  ^{3}+\left(  \frac{1}{3!}-\frac{1}%
{4!}\right)  \left(  2u^{2}\right)  ^{4}+\ldots\\
+\left(  \frac{1}{\left(  n-1\right)  !}-\frac{1}{n!}\right)  \left(
2u^{2}\right)  ^{n}+\ldots\leq\exp\left(  2u^{2}\right)  2u^{2}-2u^{2},
\end{array}
\label{control1}%
\end{equation}
then%
\begin{align*}
\underset{k\rightarrow\infty}{\lim}I_{V}\left(  u_{k}\right)   &
=\underset{k\rightarrow\infty}{\lim}\frac{\lambda}{4}\int_{\mathbb{R}^{4}%
}\left(  \exp\left(  2t_{k}^{2}w_{y_{k}}^{2}\right)  2t_{k}^{2}w_{y_{k}}%
^{2}-\left(  \exp\left(  2t_{k}^{2}w_{y_{k}}^{2}\right)  -1\right)  \right)
dx\\
&  =\underset{k\rightarrow\infty}{\lim}\frac{\lambda}{4}\int_{\mathbb{R}^{4}%
}\left(  \exp\left(  2t_{k}^{2}w^{2}\right)  2t_{k}^{2}w^{2}-\left(
\exp\left(  2t_{k}^{2}w^{2}\right)  -1\right)  \right)  dx\\
&  \leq\frac{\lambda}{4}\int_{\mathbb{R}^{4}}\left(  \exp\left(
2w^{2}\right)  2w^{2}-\left(  \exp\left(  2w^{2}\right)  -1\right)  \right)
dx=m_{\infty}.
\end{align*}
Therefore, $m_{V}<m_{\infty}$ thanks to $t_{k}<1$.
\vskip0.1cm

Next, we show $m_{V}>0$. We prove this by contradiction. Assume that there exists
some sequence $u_{k}\in$ $\mathcal{N}_{V}$ such that $I_{V}\left(
u_{k}\right)  \rightarrow0$, that is,%
\[
\frac{\lambda}{4}\int_{\mathbb{R}^{4}}\left(  \exp\left(  2u_{k}^{2}\right)
2u_{k}^{2}-\exp\left(  2u_{k}^{2}\right)  -1\right)  dx\rightarrow0.
\]
This together with (\ref{control1}) and $N_{V}\left(  u_{k}\right)  =0\ $implies
that
\[
\int_{\mathbb{R}^{4}}\exp\left(  2u_{k}^{2}\right)  2u_{k}^{2}dx\rightarrow
0\text{ and }\int_{\mathbb{R}^{4}}\left\vert \Delta u_{k}\right\vert
^{2}dx\rightarrow0.
\]

Pick $\tilde{t}_{u_{k}}>0$ such that $\tilde{t}_{u_{k}}u_{k}\in$ $\mathcal{N}%
_{\infty}$. Since $\tilde{t}_{u_{k}}$ is bounded, then $\int_{\mathbb{R}%
^{4}}\left\vert \Delta\tilde{t}_{u_{k}}u_{k}\right\vert ^{2}dx\rightarrow0$.
Hence, it follows that
\begin{align*}
I_{\infty}\left(  \tilde{t}_{u_{k}}u_{k}\right) &=\int_{\mathbb{R}^{4}%
}\left(  \exp\left(  2\tilde{t}_{u_{k}}^{2}u_{k}^{2}\right)  2\tilde{t}%
_{u_{k}}^{2}u_{k}^{2}-\left(  \exp\left(  2\tilde{t}_{u_{k}}^{2}u_{k}%
^{2}\right)  -1\right)  \right)  dx\\
&  \leq c\left(  \int_{\mathbb{R}^{4}}\left(  \exp\left(  2\tilde{t}_{u_{k}%
}^{2}u_{k}^{2}\right)  -1\right)  2\tilde{t}_{u_{k}}^{2}u_{k}^{2}dx+\tilde
{t}_{u_{k}}^{2}\int_{\mathbb{R}^{4}}\left\vert u_{k}\right\vert ^{2}dx\right)
\\
&  \leq c\tilde{t}_{u_{k}}^{2}\int_{\mathbb{R}^{4}}\left(  \left\vert \Delta
u_{k}\right\vert ^{2}+\left\vert u_{k}\right\vert ^{2}\right)  dx+\tilde
{t}_{u_{k}}^{2}\int_{\mathbb{R}^{4}}\left\vert u_{k}\right\vert ^{2}%
dx\rightarrow0,
\end{align*}
where we have used the fact that
\[
\int_{\mathbb{R}^{4}}\left(  \exp\left(  2u_{k}^{2}\right)  -1\right)
u_{k}^{2}dx\leq c\int_{\mathbb{R}^{4}}\left(  \left\vert \Delta u_{k}%
\right\vert ^{2}+\left\vert u_{k}\right\vert ^{2}\right)  dx
\]
provided that $\int_{\mathbb{R}^{4}}\left(  \left\vert \Delta u_{k}\right\vert
^{2}+\left\vert u_{k}\right\vert ^{2}\right)  dx$ is small enough. Next, we prove this fact.
Indeed, direct computation yields %
\begin{align*}
\int_{\mathbb{R}^{4}}\left(  \exp\left(  2u_{k}^{2}\right)  -1\right)
u_{k}^{2}dx  &  \leq\left(  \int_{\mathbb{R}^{4}}\left(  \exp\left(
2u_{k}^{2}\right)  -1\right)  ^{p}dx\right)  ^{1/p}\left(  \int_{\mathbb{R}%
^{4}}u_{k}^{2p^{\prime}}dx\right)  ^{1/p^{\prime}}\\
&  \leq c\left(  \int_{\mathbb{R}^{4}}\left(  \exp\left(  2pu_{k}^{2}\right)
-1\right)  dx\right)  ^{1/p}\left(  \int_{\mathbb{R}^{4}}u_{k}^{2p^{\prime}%
}dx\right)  ^{1/p^{\prime}},
\end{align*}
where $p$ and $p^{\prime}$ satisfy$\ \frac{1}{p}+\frac{1}{p^{\prime}}=1$.
\ Since $\int_{\mathbb{R}^{4}}\left(  \left\vert \Delta u_{k}\right\vert
^{2}+\left\vert u_{k}\right\vert ^{2}\right)  dx$ is small, then by the Adams'
inequality (\ref{Adams entire space}) and the Sobolev inequality, we get
\[
\int_{\mathbb{R}^{4}}\left(  \exp\left(  2u_{k}^{2}\right)  -1\right)
u_{k}^{2}dx\leq c\int_{\mathbb{R}^{4}}\left(  \left\vert \Delta u_{k}%
\right\vert ^{2}+\left\vert u_{k}\right\vert ^{2}\right)  dx.
\]
On the other hand, we have%
\[
I_{\infty}\left(  \tilde{t}_{u_{k}}u_{k}\right)  =\frac{\lambda}{4}%
\int_{\mathbb{R}^{4}}\left(  \exp\left(  2\tilde{t}_{u_{k}}^{2}u_{k}%
^{2}\right)  2\tilde{t}_{u_{k}}^{2}u_{k}^{2}-\left(  \exp\left(  2\tilde
{t}_{u_{k}}^{2}u_{k}^{2}\right)  -1\right)  \right)  dx\geq m_{\infty}.
\]
This is a contradiction. Therefore, $m_{V}>0$.
\end{proof}

\begin{lemma}
The minimizing sequence $\left\{  u_{k}\right\}  _{k}\subset\mathcal{N}%
_{V}\mathcal{\ }$is bounded in $H^{2}\left(  \mathbb{R}^{4}\right)  $.
\end{lemma}

\begin{proof}
From (\ref{control1}), we know
\begin{align*}
\int_{\mathbb{R}^{4}}\left(  \left\vert \Delta u_{k}\right\vert ^{2}+\left(
V\left(  x\right)  -\lambda\right)  \left\vert u_{k}\right\vert ^{2}\right)
dx  &  =\frac{\lambda}{2}\int_{\mathbb{R}^{4}}\left(  \exp\left(  2u_{k}%
^{2}\right)  -1\right)  2u_{k}^{2}dx\\
&  \leq\frac{c\lambda}{2}\left(  \int_{\mathbb{R}^{4}}\left(  \exp\left(
2u_{k}^{2}\right)  2u_{k}^{2}-\left(  \exp\left(  2u_{k}^{2}\right)
-1\right)  \right)  dx\right) \\
&  \rightarrow2cm_{V}.
\end{align*}
Then the proof is finished from the  assumption of $V\left(  x\right)  $.
\end{proof}
\vskip0.1cm

We now consider a minimizing sequence $\left\{  u_{k}\right\}  _{k}%
\subset\mathcal{N}_{V}$. Since the sequence is bounded in $H^{2}\left(
\mathbb{R}^{4}\right)  $, then up to a subsequence, there exists $u\in
H^{2}\left(  \mathbb{R}^{4}\right)  $, such that

\begin{itemize}
\item $u_{k}\rightarrow u$ weakly in $H^{2}\left(  \mathbb{R}^{4}\right)  $
and in $L^{p}\left(  \mathbb{R}^{4}\right)  $, for any $p>1$,

\item $u_{k}\rightarrow u$ in $L_{loc}^{p}\left(  \mathbb{R}^{4}\right)  $,

\item $u_{k}\rightarrow u$, a.e.
\end{itemize}

By extracting a subsequence, if necessary, we define $\beta$, $l\geq0$ as%

\[
\beta=\underset{k}{\lim}\int_{\mathbb{R}^{4}}\exp\left(  2u_{k}^{2}\right)
u_{k}^{2}dx\text{ and }l=\int_{\mathbb{R}^{4}}\exp\left(  2u^{2}\right)
u^{2}dx.
\]
By the weak convergence, it is obvious that $l\in\left[  0,\beta\right]  $.

\begin{lemma}
There results $\beta>0$.
\end{lemma}

\begin{proof}
We argue this by contradiction. Assume that $\beta=0$.
Since
\[
\int_{\mathbb{R}^{4}}\left(  \exp\left(  2u_{k}^{2}\right)  \right)
2u_{k}^{2}dx>c\int_{\mathbb{R}^{4}}\left(  \exp\left(  2u_{k}^{2}\right)
-1\right)  dx,
\]
 then $\int_{\mathbb{R}^{4}}\left(  \exp\left(
2u_{k}^{2}\right)  -1\right)  dx\rightarrow0$. Thus it follows that
\[
I_{V}\left(  u_{k}\right)  =\frac{\lambda}{4}\int_{\mathbb{R}^{4}}\left(
\exp\left(  2u_{k}^{2}\right)  2u_{k}^{2}-\left(  \exp\left(  2u_{k}%
^{2}\right)  -1\right)  \right)  dx\rightarrow0,
\]
which contradicts (\ref{the rang}).
\end{proof}

\begin{lemma}
If $l=\beta$, then $u\in\mathcal{N}_{V}$ and $I_{V}\left(  u\right)  =m_{V}$.
\end{lemma}

\begin{proof}
If $l=\beta$, then $\int_{\mathbb{R}^{4}}\exp\left(  2u_{k}^{2}\right)
u_{k}^{2}dx\rightarrow\int_{\mathbb{R}^{4}}\exp\left(  2u^{2}\right)  u^{2}%
dx$ as $k\rightarrow+\infty$. Then one can get
\begin{align*}
\int_{\mathbb{R}^{4}}\left(  \left\vert \Delta u\right\vert ^{2}+V\left(
x\right)  \left\vert u\right\vert ^{2}\right)  dx  &  \leq\underset
{k\rightarrow\infty}{\lim}\int_{\mathbb{R}^{4}}\left(  \left\vert \Delta
u_{k}\right\vert ^{2}+V\left(  x\right)  \left\vert u_{k}\right\vert
^{2}\right)  dx\\
&  =\underset{k\rightarrow\infty}{\lim}\lambda\int_{\mathbb{R}^{4}}\exp\left(
2u_{k}^{2}\right)  u_{k}^{2}dx\\
&  =\underset{k\rightarrow\infty}{\lim}\lambda\int_{\mathbb{R}^{4}}\exp\left(
2u^{2}\right)  u^{2}dx.
\end{align*}
If the above equality holds, then $u\in\mathcal{N}_{V}$, and the lemma is proved.
Therefore, it remains to show that
the case
\begin{equation}
\int_{\mathbb{R}^{4}}\left(  \left\vert \Delta u\right\vert ^{2}+V\left(
x\right)  \left\vert u\right\vert ^{2}\right)  dx<\lambda\int_{\mathbb{R}^{4}%
}\exp\left(  2u^{2}\right)  u^{2}dx \label{add}%
\end{equation}
cannot occur. In fact, if (\ref{add}) hold, we can take some $t\in\left(  0,1\right)  $
such that $tu\in\mathcal{N}_{V}$. Then we have
\begin{align*}
m_{\lambda}  &  \leq \frac{\lambda}{4}%
\int_{\mathbb{R}^{4}}\left(  \exp\left(  2t^{2}u^{2}\right)  2t^{2}%
u^{2}-\left(  \exp\left(  2t^{2}u^{2}\right)  -1\right)  \right)  dx\\
&  <\frac{\lambda}{4}\int_{\mathbb{R}^{4}}\left(  \exp\left(  2u^{2}\right)
2u^{2}-\left(  \exp\left(  2u^{2}\right)  -1\right)  \right)  dx\\
&  \leq\underset{k\rightarrow\infty}{\lim}\frac{\lambda}{4}\int_{\mathbb{R}%
^{4}}\left(  \exp\left(  2u_{k}^{2}\right)  2u_{k}^{2}-\left(  \exp\left(
2u_{k}^{2}\right)  -1\right)  \right)  dx=m_{\lambda},
\end{align*}
which is a contradiction.
\end{proof}

\begin{lemma}
\label{vanished all}The case $l=0$ cannot occur.
\end{lemma}

\begin{proof}
We prove this by contradiction. If $l=0$, then $u=0$, and $u_{k}\rightarrow0$
in $L_{loc}^{2}\left(  \mathbb{R}^{4}\right)  $. We first claim that:%
\begin{equation}
\int_{\mathbb{R}^{4}}\left(  \gamma-V\left(  x\right)  \right)  \left\vert
u_{k}\right\vert ^{2}dx=0. \label{vanish}%
\end{equation}

For any fixed $\varepsilon>0$, we take $R_{\varepsilon}>0$ such that
\[
\left\vert \gamma-V\left(  x\right)  \right\vert \leq\varepsilon,\text{ for
any }\left\vert x\right\vert >R_{\varepsilon}.
\]
Combining this and the boundedness of  $u_{k}$ in $H^{2}\left(  \mathbb{R}^{4}\right)$, we derive that
\begin{align*}
\int_{\mathbb{R}^{4}}\left(  \gamma-V\left(  x\right)  \right)  \left\vert
u_{k}\right\vert ^{2}dx  &  =\int_{B_{R_{\varepsilon}}}\left(  \gamma-V\left(
x\right)  \right)  \left\vert u_{k}\right\vert ^{2}dx+\int_{B_{R_{\varepsilon
}}^{c}}\left(  \gamma-V\left(  x\right)  \right)  \left\vert u_{k}\right\vert
^{2}dx\\
&  \leq c\int_{B_{R_{\varepsilon}}}\left\vert u_{k}\right\vert ^{2}dx+\gamma
M\varepsilon,
\end{align*}
where $M=\underset{k\rightarrow\infty}{\sup}\int_{\mathbb{R}^{4}}\left\vert
u_{k}\right\vert ^{2}dx$. This together with  $u_{k}\rightarrow0$ in $L_{loc}^{2}\left(
\mathbb{R}^{4}\right)  $ as $k\rightarrow\infty$ yields that
\[
\int_{\mathbb{R}^{4}}\left(  \gamma-V\left(  x\right)  \right)  \left\vert
u_{k}\right\vert ^{2}dx\leq c\varepsilon,
\]
which implies (\ref{vanish}) holds.

 By (\ref{vanish}), we see that
\[
\underset{k\rightarrow\infty}{\lim}\int_{\mathbb{R}^{4}}\left(  \left\vert
\Delta u_{k}\right\vert ^{2}+V\left(  x\right)  \left\vert u_{k}\right\vert
^{2}\right)  dx=\int_{\mathbb{R}^{4}}\left(  \left\vert \Delta u\right\vert
^{2}+\gamma\left\vert u\right\vert ^{2}\right)  dx.
\]

From the proof of Lemma \ref{impor}, we know that there exists some bounded sequence
$t_{k}\geq1$ such that $t_{k}u_{k}\in\mathcal{N}_{\infty}$, that is,%
\begin{equation}
\int_{\mathbb{R}^{4}}\left(  \left\vert \Delta u_{k}\right\vert ^{2}+\left(
\gamma-\lambda\right)  \left\vert u_{k}\right\vert ^{2}\right)  dx-\lambda
\int_{\mathbb{R}^{4}}\left(  \exp\left(  2t_{k}^{2}u_{k}^{2}\right)
-1\right)  u_{k}^{2}dx=0. \label{sub1}%
\end{equation}
On the other hand, since $ u_k \in N_{V}$, then
\begin{equation}
\int_{\mathbb{R}^{4}}\left(  \left\vert \Delta u_{k}\right\vert ^{2}+\left(
V\left(  x\right)  -\lambda\right)  \left\vert u_{k}\right\vert ^{2}\right)
dx-\lambda\int_{\mathbb{R}^{4}}\left(  \exp\left(  2u_{k}^{2}\right)
-1\right)  u_{k}^{2}dx=0.\label{sub2}%
\end{equation}

Combining (\ref{sub1}) and (\ref{sub2}),\ we get \
\begin{align*}
&  \lambda\int_{\mathbb{R}^{4}}\left(  \exp\left(  2t_{k}^{2}u_{k}^{2}\right)
-1\right)  u_{k}^{2}dx-\lambda\int_{\mathbb{R}^{4}}\left(  \exp\left(
2u_{k}^{2}\right)  -1\right)  u_{k}^{2}dx\\
&\ \  =\int_{\mathbb{R}^{4}}\left(  \left\vert \Delta u_{k}\right\vert
^{2}+\left(  \gamma-\lambda\right)  \left\vert u_{k}\right\vert ^{2}\right)
dx-\int_{\mathbb{R}^{4}}\left(  \left\vert \Delta u_{k}\right\vert
^{2}+\left(  V\left(  x\right)  -\lambda\right)  \left\vert u_{k}\right\vert
^{2}\right)  dx\\
&\ \  =\int_{\mathbb{R}^{4}}\left(  \gamma-V\left(  x\right)  \right)  \left\vert
u_{k}\right\vert ^{2}dx\rightarrow0.
\end{align*}
Hence
\begin{equation}
\int_{\mathbb{R}^{4}}\left(  \exp\left(  2t_{k}^{2}u_{k}^{2}\right)
-1\right)  u_{k}^{2}dx=\int_{\mathbb{R}^{4}}\left(  \exp\left(  2u_{k}%
^{2}\right)  -1\right)  u_{k}^{2}dx+o_{k}\left(  1\right)  . \label{converg}%
\end{equation}

Next, we claim that $t_{k}\rightarrow t_{0}=1$ as $k\rightarrow\infty$. We
prove this by contradiction. Assume that $t_{0}>1$. We carry out the proof in two cases.
\vskip0.1cm

\emph{Case 1:} There exists some $N\geq2$ such
that
\[
\underset{k\rightarrow\infty}{\lim}\int_{\mathbb{R}^{4}}u_{k}^{2N}dx>0.
\]
This will imply that
\begin{align*}
\int_{\mathbb{R}^{4}}\left(  \exp\left(  2t_{k}^{2}u_{k}^{2}\right)
-1\right)  u_{k}^{2}dx  &  \geq\int_{\mathbb{R}^{4}}\left(  \exp\left(
2t_{0}^{2}u_{k}^{2}\right)  -1\right)  u_{k}^{2}dx\\
&  >\int_{\mathbb{R}^{4}}\left(  \exp\left(  2u_{k}^{2}\right)  -1\right)
u_{k}^{2}dx>0,
\end{align*}
which is a contradiction with (\ref{converg}).
\vskip0.1cm

\emph{Case 2:} For any $N\geq2$, there holds
\[
\underset{k\rightarrow\infty}{\lim}\int_{\mathbb{R}^{4}}u_{k}^{2N}%
dx\rightarrow0.
\]
Then, we have
\begin{align*}
I_{V}\left(  u_{k}\right)   &  =\frac{\lambda}{4}\int_{\mathbb{R}^{4}}\left(
\exp\left(  2u_{k}^{2}\right)  2u_{k}^{2}-\left(  \exp\left(  2u_{k}%
^{2}\right)  -1\right)  \right)  dx\\
&  =\int_{\mathbb{R}^{4}}\left(  \frac{\left(  2u_{k}^{2}\right)  ^{2}}%
{2}+\left(  \frac{1}{2}-\frac{1}{3!}\right)  \left(  2u_{k}^{2}\right)
^{3}+\left(  \frac{1}{3!}-\frac{1}{4!}\right)  \left(  2u_{k}^{2}\right)
^{4}+\ldots+\right. \\
&  \left.  \left(  \frac{1}{\left(  n-1\right)  !}-\frac{1}{n!}\right)
\left(  2u_{k}^{2}\right)  ^{n}+\ldots\right)  dx\rightarrow\frac{\lambda}%
{2}\int_{\mathbb{R}^{4}}\exp\left(  2u_{k}^{2}\right)  u_{k}^{2}dx.
\end{align*}
Thus, it follows that $\frac{\lambda}{2}\int_{\mathbb{R}^{4}}\exp\left(  2u_{k}^{2}\right)
u_{k}^{2}dx\rightarrow m_{V}$ as $k\rightarrow+\infty$. Therefore, we conclude that
\begin{align*}
\int_{\mathbb{R}^{4}}\left(  \exp\left(  2t_{k}^{2}u_{k}^{2}\right)
-1\right)  u_{k}^{2}dx  &  \geq\int_{\mathbb{R}^{4}}\left(  \exp\left(
2t_{0}^{2}u_{k}^{2}\right)  -1\right)  u_{k}^{2}dx\\
&  \geq t_{0}^{2}\int_{\mathbb{R}^{4}}\left(  \exp\left(  2u_{k}^{2}\right)
-1\right)  u_{k}^{2}dx\\
&  \rightarrow t_{0}^{2}\frac{2m_{V}}{\lambda}>\frac{2m_{V}}{\lambda}\\
&  =\underset{k\rightarrow\infty}{\lim}\int_{\mathbb{R}^{4}}\left(
\exp\left(  2u_{k}^{2}\right)  -1\right)  u_{k}^{2}dx,
\end{align*}
which contradicts (\ref{converg}). Thus the claim is proved.\

Now, by\ (\ref{vanish}) we can obtain
\begin{align*}
m_{\infty}  &  \leq I_{\infty}\left(  t_{k}u_{k}\right) \\
&  =\frac{t_{k}^{2}}{2}\int_{\mathbb{R}^{4}}\left(  \left\vert \Delta
u_{k}\right\vert ^{2}+V\left(  x\right)  \left\vert u_{k}\right\vert
^{2}\right)  dx-\frac{\lambda}{4}\int_{\mathbb{R}^{4}}\left(  \exp\left(
2t_{k}^{2}u_{k}^{2}\right)  -1\right)  dx\ \\
&  =\frac{t_{k}^{2}}{2}\int_{\mathbb{R}^{4}}\left(  \left\vert \Delta
u_{k}\right\vert ^{2}+\gamma\left\vert u_{k}\right\vert ^{2}\right)
dx+\frac{t_{k}^{2}\varepsilon^{2}}{2}-\frac{\lambda}{4}\int_{\mathbb{R}^{4}%
}\left(  \exp\left(  2t_{k}^{2}u_{k}^{2}\right)  -1\right)  dx\\
&  \leq\frac{t_{k}^{2}}{2}\int_{\mathbb{R}^{4}}\left(  \left\vert \Delta
u_{k}\right\vert ^{2}+\gamma\left\vert u_{k}\right\vert ^{2}\right)
dx+\frac{t_{k}^{2}\varepsilon^{2}}{2}-\frac{\lambda}{4}\int_{\mathbb{R}^{4}%
}\left(  \exp\left(  2u_{k}^{2}\right)  -1\right)  dx\\
&  \rightarrow m_{V},
\end{align*}
which contradicts  (\ref{the rang}). This accomplishes the proof of Lemma \ref{vanished all}.
\end{proof}
\medskip

In the following, we consider the case $0<l<\beta$. If $0<l<\beta$, then
$u_{k}\rightarrow u\neq0$ weakly in $H^{2}\left(  \mathbb{R}^{4}\right)  $. One
can choose an increasing sequence $\left\{  R_{j}\right\}  _{j}\rightarrow +\infty$ such that
\begin{equation}
\int_{B_{R_{j}}}\exp\left(  2u^{2}\right)  u^{2}dx=l+o_{j}\left(  1\right)  ,
\label{vanishing at infinity}%
\end{equation}
and
\[
\int_{B_{R_{j}}^{c}}u^{p}dx=o_{j}\left(  1\right)  ,
\]
for any $1\leq p<\infty$. We define
\[
C_{j}=B_{R_{j}+1}\backslash B_{R_{j}}=\left\{  \left.  x\in\mathbb{R}%
^{4}\right\vert R_{j}\leq\left\vert x\right\vert <R_{j}+1\right\}  .
\]

\begin{lemma}
For the $C_{j}$ given above, we have
\begin{equation}
\int_{C_{j}}\exp\left(  2u_{k}^{2}\right)  u_{k}^{2}dx=o_{j}\left(  1\right)
\label{vani}%
\end{equation}
and%
\begin{equation}
\int_{C_{j}}\left\vert \Delta u_{k}\right\vert ^{2}dx=o_{j}\left(  1\right)  .
\label{vani2}%
\end{equation}

\end{lemma}

\begin{proof}
We prove (\ref{vani}) by contradiction. If there exists some subsequence
$\left\{  j_{i}\right\}  _{i}$ of $\left\{  j\right\}  $ such that
(\ref{vani})fails, then we must have
\[
\underset{i=1}{\overset{\infty}{\sum}}\int_{C_{j_{i}}}\exp\left(  2u_{k}%
^{2}\right)  u_{k}^{2}dx=\infty.
\]
However, on the other hand,
\begin{align*}
\underset{i=1}{\overset{\infty}{\sum}}\int_{C_{j_{i}}}\exp\left(  2u_{k}%
^{2}\right)  u_{k}^{2}dx  &  \leq\int_{\mathbb{R}^{4}}\exp\left(  2u_{k}%
^{2}\right)  u_{k}^{2}dx\\
&  =\frac{1}{\lambda}\int_{\mathbb{R}^{4}}\left(  \left\vert \Delta
u_{k}\right\vert ^{2}+V\left(  x\right)  \left\vert u_{k}\right\vert
^{2}\right)  dx<\infty,
\end{align*}
which arrives at a contradiction.  Similarly, one can also prove \ (\ref{vani2}).
\end{proof}

\begin{lemma}
\label{strong}There holds%
\[
\underset{k\rightarrow\infty}{\lim}\int_{\mathbb{R}^{4}}\left(  \exp\left(
2u_{k}^{2}\right)  -1\right)  dx=\int_{\mathbb{R}^{4}}\left(  \exp\left(
2u^{2}\right)  -1\right)  dx.
\]

\end{lemma}

\begin{proof}
For any $R>0$, we can write
\begin{align*}
&  \left\vert \int_{\mathbb{R}^{4}}\left(  \exp\left(  2u_{k}^{2}\right)
-1\right)  dx-\int_{\mathbb{R}^{4}}\left(  \exp\left(  2u^{2}\right)
-1\right)  dx\right\vert \\
& \ \ =\left\vert \int_{\mathbb{R}^{4}\cap\left\{  u_{k}<R\right\}  }\left(
\exp\left(  2u_{k}^{2}\right)  -1\right)  dx-\int_{\mathbb{R}^{4}\cap\left\{
u_{k}<R\right\}  }\left(  \exp\left(  2u^{2}\right)  -1\right)  dx\right\vert
\\
& \ \ \ \  +\left\vert \int_{\mathbb{R}^{4}\cap\left\{  u_{k}\geq R\right\}  }\left(
\exp\left(  2u_{k}^{2}\right)  -1\right)  dx-\int_{\mathbb{R}^{4}\cap\left\{
u_{k}\geq R\right\}  }\left(  \exp\left(  2u^{2}\right)  -1\right)
dx\right\vert \\
&  =I_{k,R}+II_{k,R}.
\end{align*}
A direct application of the dominated convergence theorem leads to $I_{k,R}\rightarrow0$. For
$II_{k,R}$, we have
\begin{align*}
\int_{\mathbb{R}^{4}\cap\left\{  u_{k}\geq R\right\}  }\left(  \exp\left(
2u_{k}^{2}\right)  -1\right)  dx&\leq\frac{1}{R^{2}}\int_{\mathbb{R}^{4}%
\cap\left\{  u_{k}\geq R\right\}  }u_{k}^{2}\left(  \exp\left(  2u_{k}%
^{2}\right)  -1\right)  dx\\
&  \leq\frac{1}{R^{2}}\int_{\mathbb{R}^{4}}u_{k}^{2}\exp\left(  2u_{k}%
^{2}\right)  dx\rightarrow0,\text{ as }R\rightarrow\infty,
\end{align*}
where we have used the fact that $\int_{\mathbb{R}^{4}}u_{k}^{2}\exp\left(
2u_{k}^{2}\right)  dx$ is bounded. Consequently, $II_{k,R}\rightarrow0$, and the
proof is finished.
\end{proof}

\begin{lemma}
\label{strong copy 2}If $\int_{\mathbb{R}^{4}}\left(  \left\vert \Delta
u\right\vert ^{2}+V\left(  x\right)  \left\vert u\right\vert ^{2}\right)
dx>\lambda\int_{\mathbb{R}^{4}}\exp\left(  2u^{2}\right)  u^{2}dx$, then
\[
\underset{k\rightarrow\infty}{\lim}\int_{B_{R_{j}}}\exp\left(  2u_{k}%
^{2}\right)  u_{k}^{2}dx=\int_{B_{R_{j}}}\exp\left(  2u^{2}\right)  u^{2}dx
\]
provided $j$ large enough.
\end{lemma}

\begin{proof}
Obviously,
\[
\underset{k}{\lim\inf}\int_{B_{R_{j}}}\left(  \left\vert \Delta u_{k}%
\right\vert ^{2}+V\left(  x\right)  \left\vert u_{k}\right\vert ^{2}\right)
dx\geq\int_{B_{R_{j}}}\left(  \left\vert \Delta u\right\vert ^{2}+V\left(
x\right)  \left\vert u\right\vert ^{2}\right)  dx
\] through lower semi-continuity.
We split the proof into two cases. \
\vskip0.1cm

\emph{Case 1:} If
\[
\underset{k}{\lim}\int_{B_{R_{j}}}\left(  \left\vert \Delta u_{k}\right\vert
^{2}+V\left(  x\right)  \left\vert u_{k}\right\vert ^{2}\right)
dx=\int_{B_{R_{j}}}\left(  \left\vert \Delta u\right\vert ^{2}+V\left(
x\right)  \left\vert u\right\vert ^{2}\right)  dx,
\]
then by the Adams' inequality (\ref{Adams entire space}), we know that for any $p>1$,%
\[%
\begin{array}
[c]{l}%
\int_{B_{R_{j}}}\left(  \exp\left(  2u_{k}^{2}\right)  u_{k}^{2}\right)
^{p}dx\leq c\int_{B_{R_{j}}}\left(  \exp\left(  2pu_{k}^{2}\right)  -1\right)
u_{k}^{2p}dx+\int_{B_{R_{j}}}u_{k}^{2p}dx\\
\leq\left(  \int_{B_{R_{j}}}\left(  \exp\left(  4pu_{k}^{2}\right)  -1\right)
dx\right)  ^{1/2}\left(  \int_{B_{R_{j}}}u_{k}^{4p}dx\right)  ^{1/2}%
+\int_{B_{R_{j}}}u_{k}^{2p}dx<+\infty.
\end{array}
\]
Therefore,
\[
\underset{k\rightarrow\infty}{\lim}\int_{B_{R_{j}}}\left(  \exp\left(
2u_{k}^{2}\right)  u_{k}^{2}\right)  dx=\int_{B_{R_{j}}}\left(  \exp\left(
2u^{2}\right)  u^{2}\right)  dx.
\]
\vskip0.1cm

\emph{Case 2:} If $\underset{k}{\lim}\int_{B_{R_{j}}}\left(  \left\vert \Delta
u_{k}\right\vert ^{2}+V\left(  x\right)  \left\vert u_{k}\right\vert
^{2}\right)  dx>\int_{B_{R_{j}}}\left(  \left\vert \Delta u\right\vert
^{2}+V\left(  x\right)  \left\vert u\right\vert ^{2}\right)  dx$, we set
\begin{align*}
v_{k}  &  =\frac{u_{k}}{\underset{k}{\lim}\int_{B_{R_{j}}}\left(  \left\vert
\Delta u_{k}\right\vert ^{2}+V\left(  x\right)  \left\vert u_{k}\right\vert
^{2}\right)  dx},\text{ and }\\
\text{ }v_{0}  &  =\frac{u}{\underset{k}{\lim}\int_{B_{R_{j}}}\left(
\left\vert \Delta u_{k}\right\vert ^{2}+V\left(  x\right)  \left\vert
u_{k}\right\vert ^{2}\right)  dx}.
\end{align*}

We claim that there exists some $q>1$ sufficiently closed to $1$ such that
\begin{equation}
q\underset{k}{\lim}\int_{B_{R_{j}}}\left(  \left\vert \Delta u_{k}\right\vert
^{2}+V\left(  x\right)  \left\vert u_{k}\right\vert ^{2}\right)
dx<\frac{16\pi^{2}}{1-\int_{B_{R_{j}}}\left(  \left\vert \Delta v_{0}%
\right\vert ^{2}+V\left(  x\right)  \left\vert v_{0}\right\vert ^{2}\right)
dx}.\label{eqa}%
\end{equation}
Indeed, we have
\begin{align*}
&  \int_{B_{R_{j}}}\left(  \left\vert \Delta u_{k}\right\vert ^{2}+V\left(
x\right)  \left\vert u_{k}\right\vert ^{2}\right)  dx\cdot\left(
1-\int_{B_{R_{j}}}\left(  \left\vert \Delta v_{0}\right\vert ^{2}+V\left(
x\right)  \left\vert v_{0}\right\vert ^{2}\right)  dx\right)  \\
& \ \  \leq\int_{\mathbb{R}^{4}}\left(  \left\vert \Delta u_{k}\right\vert
^{2}+V\left(  x\right)  \left\vert u_{k}\right\vert ^{2}\right)
dx-\int_{B_{R_{j}}}\left(  \left\vert \Delta u\right\vert ^{2}+V\left(
x\right)  \left\vert u\right\vert ^{2}\right)  dx\\
&\ \   =\frac{\lambda}{2}\int_{\mathbb{R}^{4}}\left(  \exp\left(  2u_{k}%
^{2}\right)  -1\right)  dx-\frac{\lambda}{2}\int_{\mathbb{R}^{4}}\left(
\exp\left(  2u^{2}\right)  -1\right)  dx+\\
& \ \ \ \  +2I_{V}\left(  u_{k}\right)  -2I_{V}\left(  u\right)  +o_{j}\left(
1\right)  .
\end{align*}

Since $\int_{\mathbb{R}^{4}}\left(  \left\vert \Delta u\right\vert
^{2}+V\left(  x\right)  \left\vert u\right\vert ^{2}\right)  dx>\lambda
\int_{\mathbb{R}^{4}}\exp\left(  2u^{2}\right)  u^{2}dx$, then it follows that
\begin{equation}
I_{V}\left(  u\right)  >\frac{\lambda}{4}\int_{\mathbb{R}^{4}}\left(
\exp\left(  2u^{2}\right)  2u^{2}-\left(  \exp\left(  2u^{2}\right)
-1\right)  \right)  dx>0.\label{adda}%
\end{equation}
Combining (\ref{adda}) and the fact (Lemma \ref{lem5}), we conclude that
\[
\underset{k\rightarrow\infty}{\lim}I_{V}\left(  u_{k}\right)  =m_{V}%
<m_{\infty}<8\pi^{2}.
\]
This proves the claim.
\vskip0.1cm

By the Concentration-Compactness principle for Adams' inequality on
$H^{2}\left(  \mathbb{R}^{4}\right)  $, there exists some $p_{0}>1$ such that
\[
\underset{k}{\sup}\int_{B_{R_{j}}}\left(  \exp\left(  2u_{k}^{2}\right)
-1\right)  ^{p_{0}}dx<\infty\text{. }%
\]
Then it follows that there exists some $\tilde{p}_{0}>1$ such that
\[
\underset{k}{\sup}\int_{B_{R_{j}}}\left(  \exp\left(  2u_{k}^{2}\right)
u_{k}^{2}\right)  ^{\tilde{p}_{0}}dx<\infty.
\]
Therefore, we get
\[
\underset{k\rightarrow\infty}{\lim}\int_{B_{R_{j}}}\left(  \exp\left(
2u_{k}^{2}\right)  u_{k}^{2}\right)  dx=\int_{B_{R_{j}}}\exp\left(
2u^{2}\right)  u^{2}dx.
\]

\end{proof}

From above, we can extract a subsequence $u_{k_{j}}$ such that for every $j\in%
\mathbb{N}
$,%
\[
\int_{C_{j}}\exp\left(  2u_{k_{j}}^{2}\right)  u_{k_{j}}^{2}dx=o_{j}\left(
1\right)  ,\int_{C_{j}}\left\vert \Delta u_{k_{j}}\right\vert ^{2}%
dx=o_{j}\left(  1\right)  ,\int_{C_{j}}u_{k_{j}}^{2}dx=o_{j}\left(  1\right).
\]
We take $\left\{  u_{k_{j}}\right\}  $ as a new minimizing sequence renaming
it $\left\{  u_{j}\right\}  _{j}$.

Now, for every $j$, we define a function $\psi_{j}\in C_{c}^{\infty}\left(
\mathbb{R}^{4}\right)$ satisfying $0\leq\psi_{j}\left(  x\right)  \leq1$, $\psi_{j}\left(  x\right)  =1$ if $\left\vert x\right\vert \leq
R_{j}$, $\psi_{j}\left(  x\right)  =0$ if $\left\vert x\right\vert >R_{j}+1$,
and $\left\vert \nabla\psi_{j}\left(  x\right)  \right\vert ,\left\vert
\Delta\psi_{j}\left(  x\right)  \right\vert \leq c$ for every $x$. We also
define auxiliary functions
\[
u_{j}^{\prime}=\psi_{j}u_{j}\text{, and }u_{j}^{\prime\prime}=\left(
1-\psi_{j}\right)  u_{j}.
\]
Obviously, we have $u_{j}=u_{j}^{\prime}+u_{j}^{\prime\prime}$ for every $j$.

\begin{lemma}
\label{spit lemma}The following properties hold as $j\rightarrow\infty:$

1. $u_{j}^{\prime}\rightarrow u$ weakly in $H^{2}\left(  \mathbb{R}%
^{4}\right)  $, strongly in $L^{p}\left(  \mathbb{R}^{4}\right)  $ for any
$1\leq p<\infty$, and $u_{j}^{\prime\prime}\rightarrow0$ weakly in
$H^{2}\left(  \mathbb{R}^{4}\right)  $.
\vskip0.1cm

2. There results
\begin{align}
&  \int_{\mathbb{R}^{4}}\exp\left(  2u_{j}^{2}\right)  u_{j}^{2}dx\nonumber\\
&  \ \ =\int_{\mathbb{R}^{4}}\exp\left(  2\left(  u_{j}^{\prime}\right)
^{2}\right)  \left(  u_{j}^{\prime}\right)  ^{2}dx+\int_{\mathbb{R}^{4}}%
\exp\left(  2\left(  u_{j}^{\prime\prime}\right)  ^{2}\right)  \left(
u_{j}^{\prime\prime}\right)  ^{2}dx+o_{j}\left(  1\right)  . \label{split1}%
\end{align}
\vskip0.1cm

3. There results
\begin{align}
&  \int_{\mathbb{R}^{4}}\left(  \left\vert \Delta u_{j}\right\vert
^{2}+V\left(  x\right)  \left\vert u_{j}\right\vert ^{2}\right)  dx\nonumber\\
& \ \ =\int_{\mathbb{R}^{4}}\left(  \left\vert \Delta u_{j}^{\prime}\right\vert
^{2}+V\left(  x\right)  \left\vert u_{j}^{\prime}\right\vert ^{2}\right)
dx+\int_{\mathbb{R}^{4}}\left(  \left\vert \Delta u_{j}^{\prime\prime
}\right\vert ^{2}+V\left(  x\right)  \left\vert u_{j}^{\prime\prime
}\right\vert ^{2}\right)  dx+o_{j}\left(  1\right)  . \label{split2}%
\end{align}

\end{lemma}

\begin{proof}
The first property is obvious, by the defnitions of $u_{j}^{\prime}$ and
$u_{j}^{\prime\prime}$.\ \ Now, we check the second equality. By (\ref{vani}), we derive that
\begin{align*}
&  \int_{\mathbb{R}^{4}}\exp\left(  2u_{j}^{2}\right)  u_{j}^{2}dx\\
&  =\int_{B_{R_{j}}}\exp\left(  2u_{j}^{2}\right)  u_{j}^{2}dx+\int_{C_{j}%
}\exp\left(  2u_{j}^{2}\right)  u_{j}^{2}dx+\int_{B_{R_{j}}^{c}+1}\exp\left(
2u_{j}^{2}\right)  u_{j}^{2}dx\\
&  =\int_{B_{R_{j}}}\exp\left(  2\left(  u_{j}^{\prime}\right)  ^{2}\right)
\left(  u_{j}^{\prime}\right)  ^{2}dx+\int_{C_{j}}\exp\left(  2u_{j}%
^{2}\right)  u_{j}^{2}dx+\int_{B_{R_{j}}^{c}+1}\exp\left(  2\left(
u_{j}^{\prime\prime}\right)  ^{2}\right)  \left(  u_{j}^{\prime\prime}\right)
^{2}dx\\
&  =\int_{\mathbb{R}^{4}}\exp\left(  2\left(  u_{j}^{\prime}\right)
^{2}\right)  \left(  u_{j}^{\prime}\right)  ^{2}dx+\int_{\mathbb{R}^{4}}%
\exp\left(  2\left(  u_{j}^{\prime\prime}\right)  ^{2}\right)  \left(
u_{j}^{\prime\prime}\right)  ^{2}dx+o_{j}\left(  1\right).
\end{align*}

We now prove the third property. Since $q\left(  x\right)  >0$, direct computation leads to
\begin{align*}
\int_{\mathbb{R}^{4}}V\left(  x\right)  \left\vert u_{j}\right\vert ^{2}dx  &
=\int_{B_{R_{j}}}V\left(  x\right)  \left\vert u_{j}^{\prime}\right\vert
^{2}dx+\int_{B_{R_{j}}^{c}+1}V\left(  x\right)  \left\vert u_{j}^{\prime
\prime}\right\vert ^{2}dx+\int_{C_{j}}V\left(  x\right)  \left\vert
u_{j}\right\vert ^{2}dx\\
&  =\int_{B_{R_{j}}}V\left(  x\right)  \left\vert u_{j}^{\prime}\right\vert
^{2}dx+\int_{B_{R_{j}}^{c}+1}V\left(  x\right)  \left\vert u_{j}^{\prime
\prime}\right\vert ^{2}dx+o_{j}\left(  1\right) \\
&  =\int_{\mathbb{R}^{4}}V\left(  x\right)  \left\vert u_{j}^{\prime
}\right\vert ^{2}dx+\int_{\mathbb{R}^{4}}V\left(  x\right)  \left\vert
u_{j}^{\prime\prime}\right\vert ^{2}dx+o_{j}\left(  1\right).
\end{align*}%
We now only need to show that
$$\int_{\mathbb{R}^{4}}|\Delta u_{j}|^2dx=\int_{\mathbb{R}^{4}}|\Delta u_{j}^{\prime}|^2dx+\int_{\mathbb{R}^{4}}|\Delta u_{j}^{\prime\prime}|^2dx+o_j(1).$$
Observing
\begin{align*}
\int_{\mathbb{R}^{4}}\left\vert \Delta u_{j}\right\vert ^{2}dx  &
=\int_{\mathbb{R}^{4}}\left\vert \Delta u_{j}^{\prime}+\Delta u_{j}%
^{\prime\prime}\right\vert ^{2}dx\\
&  =\int_{\mathbb{R}^{4}}\left\vert \Delta u_{j}^{\prime}\right\vert
^{2}dx+\int_{\mathbb{R}^{4}}\left\vert \Delta u_{j}^{\prime\prime}\right\vert
^{2}dx+\int_{\mathbb{R}^{4}}\Delta u_{j}^{\prime}\cdot\Delta u_{j}%
^{\prime\prime}dx,
\end{align*}
in order to obtain the desired result, we only need to verify that
$$\int_{\mathbb{R}^{4}}\Delta u_{j}^{\prime}\cdot\Delta u_{j}%
^{\prime\prime}dx=o_j(1).$$
We can write
\begin{align*}
&  \int_{\mathbb{R}^{4}}\Delta u_{k}^{\prime}\cdot\Delta u_{k}^{\prime\prime
}dx\\
&\ \   =\int_{\mathbb{R}^{4}}\left(  u_{j}\Delta\psi_{j}+u_{j}\Delta\psi
_{j}+2\nabla u_{j}\nabla\psi_{j}\right)  \cdot\left(  u_{j}\Delta\left(
1-\psi_{j}\right)  \right. \\
& \ \ \ \  \left.  +\left(  1-\psi_{j}\right)  \Delta u_{j}+2\nabla\left(  1-\psi
_{j}\right)  \nabla u_{j}\right)  dx\\
& \ \  =\int_{\mathbb{R}^{4}}\left(  u_{j}\Delta\psi_{j}+u_{j}\Delta\psi
_{j}+2\nabla u_{j}\nabla\psi_{j}\right)  \cdot\left(  -u_{j}\Delta\psi
_{j}+\left(  1-\psi_{j}\right)  \Delta u_{j}-2\nabla\psi_{j}\nabla
u_{j}\right)  dx\\
& \ \  =\int_{\mathbb{R}^{4}}\left(  -\left\vert u_{j}\right\vert ^{2}\left\vert
\Delta\psi_{j}\right\vert ^{2}+\left(  1-\psi_{j}\right)  u_{j}\Delta\psi
_{j}\Delta u_{j}-2u_{j}\Delta\psi_{j}\nabla\psi_{j}\nabla u_{j}\right)  dx\\
& \ \ \ \  +\int_{\mathbb{R}^{4}}\left(  -\left\vert u_{j}\right\vert ^{2}\left\vert
\Delta\psi_{j}\right\vert ^{2}+\left(  1-\psi_{j}\right)  u_{j}\Delta
u_{j}\Delta\psi_{j}-2u_{j}\nabla u_{j}\nabla\psi_{j}\Delta\psi_{j}\right)
dx\\
& \ \ \ \  +\int_{\mathbb{R}^{4}}\left(  -2\nabla u_{j}\nabla\psi_{j}u_{j}\Delta
\psi_{j}+2\nabla u_{j}\nabla\psi_{j}\left(  1-\psi_{j}\right)  \Delta
u_{j}-4\left\vert \nabla u_{j}\right\vert ^{2}\left\vert \nabla\psi
_{j}\right\vert ^{2}\right)  dx\\
& \ \  =I+II+III.
\end{align*}
For $I$, we have
\begin{align*}
I&=\int_{\mathbb{R}^{4}}\left\vert \left(  -\left\vert u_{j}\right\vert
^{2}\left\vert \Delta\psi_{j}\right\vert ^{2}+\left(  1-\psi_{j}\right)
u_{j}\Delta\psi_{j}\Delta u_{j}-2u_{j}\Delta\psi_{j}\nabla\psi_{j}\nabla
u_{j}\right)  \right\vert dx\\
&  \leq c\int_{C_{j}}\left\vert u_{j}\right\vert ^{2}dx+\int_{C_{j}}\left\vert
u_{j}\right\vert \left\vert \Delta u_{j}\right\vert dx+c\int_{C_{j}}\left\vert
u_{j}\nabla u_{j}\right\vert dx\leq\\
&  \leq c\int_{C_{j}}\left\vert u_{j}\right\vert ^{2}dx+\left(  \int_{C_{j}%
}\left\vert u_{j}\right\vert ^{2}dx\right)  ^{1/2}\left(  \int_{C_{j}%
}\left\vert \Delta u_{j}\right\vert ^{2}dx\right)  ^{1/2}\\
& \ \  +c\left(  \int_{C_{j}}\left\vert u_{j}\right\vert ^{2}dx\right)
^{1/2}\left(  \int_{C_{j}}\left\vert \nabla u_{j}\right\vert ^{2}dx\right)
^{1/2}=o_{j}\left(  1\right).
\end{align*}%
For $II$, we derive that
\begin{align}
II&=\int_{\mathbb{R}^{4}}\left\vert \left(  -\left\vert u_{j}\right\vert
^{2}\left\vert \Delta\psi_{j}\right\vert ^{2}+\left(  1-\psi_{j}\right)
u_{j}\Delta u_{j}\Delta\psi_{j}-2u_{j}\nabla u_{j}\nabla\psi_{j}\Delta\psi
_{j}\right)  \right\vert dx\nonumber\\
&  \leq\int_{\mathbb{R}^{4}}\left\vert u_{j}\right\vert ^{2}\left\vert
\Delta\psi_{j}\right\vert ^{2}dx+\int_{\mathbb{R}^{4}}\left\vert u_{j}\Delta
u_{j}\Delta\psi_{j}\right\vert dx+\int_{\mathbb{R}^{4}}\left\vert 2u_{j}\nabla
u_{j}\nabla\psi_{j}\Delta\psi_{j}\right\vert dx\nonumber\\
&  \leq c\int_{C_{j}}\left\vert u_{j}\right\vert ^{2}dx+c\left(  \int_{C_{j}%
}\left\vert u_{j}\right\vert ^{2}dx\right)  ^{1/2}\left(  \int_{C_{j}%
}\left\vert \Delta u_{j}\right\vert ^{2}dx\right)  ^{1/2}\nonumber\\
& \ \  +c\left(  \int_{C_{j}}\left\vert u_{j}\right\vert ^{2}dx\right)
^{1/2}\left(  \int_{C_{j}}\left\vert \nabla u_{j}\right\vert ^{2}dx\right)
^{1/2}=o_{j}\left(  1\right). \label{1}%
\end{align}
For $III$, obviously we have

\begin{align}
III&=\int_{\mathbb{R}^{4}}\left(  -2\nabla u_{j}\nabla\psi_{j}u_{j}%
\Delta\psi_{j}+2\nabla u_{j}\nabla\psi_{j}\left(  1-\psi_{j}\right)  \Delta
u_{j}-4\left\vert \nabla u_{j}\right\vert ^{2}\left\vert \nabla\psi
_{j}\right\vert ^{2}\right)  dx\nonumber\\
&  =2\int_{\mathbb{R}^{4}}\left\vert \nabla u_{j}\nabla\psi_{j}u_{j}\Delta
\psi_{j}\right\vert dx+2\int_{\mathbb{R}^{4}}\left\vert \nabla u_{j}\nabla
\psi_{j}\left(  1-\psi_{j}\right)  \Delta u_{j}\right\vert dx+4\int
_{\mathbb{R}^{4}}\left\vert \nabla u_{j}\right\vert ^{2}\left\vert \nabla
\psi_{j}\right\vert ^{2}dx\nonumber\\
&  \leq c\left(  \int_{C_{j}}\left\vert u_{j}\right\vert ^{2}dx\right)
^{1/2}\left(  \int_{C_{j}}\left\vert \nabla u_{j}\right\vert ^{2}dx\right)
^{1/2}+c\left(  \int_{C_{j}}\left\vert \nabla u_{j}\right\vert ^{2}dx\right)
^{1/2}\cdot\nonumber\\
&  \cdot\left(  \int_{C_{j}}\left\vert \Delta u_{j}\right\vert ^{2}dx\right)
^{1/2}+c\int_{C_{j}}\left\vert \nabla u_{j}\right\vert ^{2}dx. \label{3}%
\end{align}
By using the Sobolev interpolation inequality, we get
\begin{equation}
\int_{C_{j}}\left\vert \nabla u_{j}\right\vert ^{2}dx\leq c\left(  \int
_{C_{j}}\left\vert u_{j}\right\vert ^{2}dx+\int_{C_{j}}\left\vert \Delta
u_{j}\right\vert ^{2}dx\right)  =o_{j}\left(  1\right). \label{4}%
\end{equation}
Combining (\ref{1})-(\ref{4}), we finish the proof.
\end{proof}

\begin{lemma}
There holds%
\[
\frac{\lambda}{4}\int_{\mathbb{R}^{4}}\left(  \exp\left(  2u^{2}\right)
2u^{2}-\left(  \exp\left(  2u^{2}\right)  -1\right)  \right)  dx\leq m_{V}.
\]

\end{lemma}

\begin{proof}
Since $u_{j}\rightarrow u$ weakly in $H^{2}\left(  \mathbb{R}^{4}\right)  $
and in $L^{p}\left(  \mathbb{R}^{4}\right)  $ for any $p>1$, thus%
\begin{align*}
&  \frac{\lambda}{4}\int_{\mathbb{R}^{4}}\left(  \exp\left(  2u^{2}\right)
2u^{2}-\left(  \exp\left(  2u^{2}\right)  -1\right)  \right)  dx\\
& \ \  \leq\frac{\lambda}{4}\underset{j\rightarrow+\infty}{\lim}\int
_{\mathbb{R}^{4}}\left(  \exp\left(  2u_{j}^{2}\right)  2u_{j}^{2}-\left(
\exp\left(  2u_{j}^{2}\right)  -1\right)  \right)  dx\\
& \ \  =m_{V}.
\end{align*}

\end{proof}

\begin{lemma}
It cannot be
\begin{equation}
\int_{\mathbb{R}^{4}}\left(  \left\vert \Delta u\right\vert ^{2}+V\left(
x\right)  \left\vert u\right\vert ^{2}\right)  dx<\lambda\int_{\mathbb{R}^{4}%
}\exp\left(  2u^{2}\right)  u^{2}dx. \label{small}%
\end{equation}

\end{lemma}

\begin{proof}
If (\ref{small}) is true, then there exists some $t\in\left(  0,1\right)  $
such that $tu\in\mathcal{N}_{V}$. Therefore,
\begin{align*}
m_{V}  &  \leq I_{V}\left(  tu\right)  =\frac{\lambda}{4}\int_{\mathbb{R}^{4}%
}\left(  \exp\left(  2t^{2}u^{2}\right)  2t^{2}u^{2}-\left(  \exp\left(
2t^{2}u^{2}\right)  -1\right)  \right)  dx\\
&  \leq\frac{\lambda}{4}\int_{\mathbb{R}^{4}}\left(  \exp\left(
2u^{2}\right)  2u^{2}-\left(  \exp\left(  2u^{2}\right)  -1\right)  \right)
dx\\
&  <I_{V}\left(  u\right)  \leq m_{V},
\end{align*}
which is a contradiction.
\end{proof}

\begin{lemma}\label{adxin1}
It cannot be
\begin{equation}
\int_{\mathbb{R}^{4}}\left(  \left\vert \Delta u\right\vert ^{2}+V\left(
x\right)  \left\vert u\right\vert ^{2}\right)  dx>\lambda\int_{\mathbb{R}^{4}%
}\exp\left(  2u^{2}\right)  u^{2}dx. \label{large}%
\end{equation}

\end{lemma}

\begin{proof}
By Lemma \ref{spit lemma}, we get
\begin{align*}
&  \int_{\mathbb{R}^{4}}\left(  \left\vert \Delta u_{j}^{\prime}\right\vert
^{2}+V\left(  x\right)  \left\vert u_{j}^{\prime}\right\vert ^{2}\right)
dx+\int_{\mathbb{R}^{4}}\left(  \left\vert \Delta u_{j}^{\prime\prime
}\right\vert ^{2}+V\left(  x\right)  \left\vert u_{j}^{\prime\prime
}\right\vert ^{2}\right)  dx\\
&\ \   =\int_{\mathbb{R}^{4}}\left(  \left\vert \Delta u_{j}\right\vert
^{2}+V\left(  x\right)  \left\vert u_{j}\right\vert ^{2}\right)
dx+o_{j}\left(  1\right)  =\lambda\int_{\mathbb{R}^{4}}\exp\left(  2u_{j}%
^{2}\right)  u_{j}^{2}dx+o_{j}\left(  1\right) \\
& \ \  =\lambda\int_{\mathbb{R}^{4}}\exp\left(  2\left(  u_{j}^{\prime}\right)
^{2}\right)  \left(  u_{j}^{\prime}\right)  ^{2}dx+\lambda\int_{\mathbb{R}%
^{4}}\exp\left(  2\left(  u_{j}^{\prime\prime}\right)  ^{2}\right)  \left(
u_{j}^{\prime\prime}\right)  ^{2}dx+o_{j}\left(  1\right)  .
\end{align*}

Assume for contradiction that (\ref{large}) holds, then we can pick some
$\delta>0$ such that
\begin{equation}\label{ad2}
\int_{\mathbb{R}^{4}}\left(  \left\vert \Delta u\right\vert ^{2}+V\left(
x\right)  \left\vert u\right\vert ^{2}\right)  dx>\lambda\int_{\mathbb{R}^{4}%
}\exp\left(  2u^{2}\right)  u^{2}dx+\delta.
\end{equation}
Since $u_{j}^{\prime}\rightarrow u$ weakly in $H^{2}\left(  \mathbb{R}%
^{4}\right)  $,$\ $by (\ref{large}) and Lemma \ref{strong copy 2}, we have
\begin{align*}
&  \underset{k\rightarrow+\infty}{\lim\inf}\int_{\mathbb{R}^{4}}\left(  \left\vert \Delta
u_{j}^{\prime}\right\vert ^{2}+V\left(  x\right)  \left\vert u_{j}^{\prime
}\right\vert ^{2}\right)  dx\\
& \ \  \geq\int_{\mathbb{R}^{4}}\left(  \left\vert \Delta u\right\vert
^{2}+V\left(  x\right)  \left\vert u\right\vert ^{2}\right)  dx\\
& \ \  >\lambda\int_{\mathbb{R}^{4}}\exp\left(  2u^{2}\right)  u^{2}dx+\delta\\
&\ \   =\lambda\int_{\mathbb{R}^{4}}\exp\left(  2\left(  u_{j}^{\prime}\right)
^{2}\right)  \left(  u_{j}^{\prime}\right)  ^{2}dx+\delta+o_{j}\left(
1\right)  .
\end{align*}
Hence, we have
\[
\int_{\mathbb{R}^{4}}\left(  \left\vert \Delta u_{j}^{\prime\prime}\right\vert
^{2}+V\left(  x\right)  \left\vert u_{j}^{\prime\prime}\right\vert
^{2}\right)  dx<\lambda\int_{\mathbb{R}^{4}}\exp\left(  2\left(  u_{j}%
^{\prime\prime}\right)  ^{2}\right)  \left(  u_{j}^{\prime\prime}\right)
^{2}dx-\delta+o_{j}\left(  1\right)
\]
for $j$ large enough. \ Since $u_{j}^{\prime\prime}\rightarrow0$, weakly in
$H^{2}\left(  \mathbb{R}^{4}\right)  $, and arguing as Lemma
\ref{vanished all}, we can obtain
\[
\underset{j\rightarrow+\infty}{\lim}\int_{\mathbb{R}^{4}}\left(  \left\vert \Delta u_{j}%
^{\prime\prime}\right\vert ^{2}+V\left(  x\right)  \left\vert u_{j}%
^{\prime\prime}\right\vert ^{2}\right)  dx=\int_{\mathbb{R}^{4}}\left(
\left\vert \Delta u_{j}^{\prime\prime}\right\vert ^{2}+\gamma\left\vert
u_{j}^{\prime\prime}\right\vert ^{2}\right)  dx.
\]
Therefore, it follows that for $j$ large enough, there holds
\[
\int_{\mathbb{R}^{4}}\left(  \left\vert \Delta u_{j}^{\prime\prime}\right\vert
^{2}+\gamma\left\vert u_{j}^{\prime\prime}\right\vert ^{2}\right)
dx<\lambda\int_{\mathbb{R}^{4}}\exp\left(  2\left(  u_{j}^{\prime\prime
}\right)  ^{2}\right)  \left(  u_{j}^{\prime\prime}\right)  ^{2}%
dx-\delta+o_{j}\left(  1\right).
\]
\ By the usual argument, we can find some $t_{j}\in\left(  0,1\right)  $ such that
$t_{j}u_{j}^{\prime\prime}\in\mathcal{N}_{\infty}$, so we conclude that
\begin{align*}
m_{\infty}  &  \leq I_{\infty}\left(  t_{j}u_{j}^{\prime\prime}\right)  \leq
I_{\infty}\left(  u_{j}^{\prime\prime}\right) \\
&  =\frac{\lambda}{4}\int_{\mathbb{R}^{4}}\left(  \exp\left(  2\left\vert
u_{j}^{\prime\prime}\right\vert ^{2}\right)  2\left\vert u_{j}^{\prime\prime
}\right\vert ^{2}-\left(  \exp\left(  2\left\vert u_{j}^{\prime\prime
}\right\vert ^{2}\right)  -1\right)  \right)  dx\\
&  \leq\frac{\lambda}{4}\int_{\mathbb{R}^{4}}\left(  \exp\left(  2\left\vert
u_{j}^{\prime\prime}\right\vert ^{2}\right)  2\left\vert u_{j}^{\prime\prime
}\right\vert ^{2}-\left(  \exp\left(  2\left\vert u_{j}^{\prime\prime
}\right\vert ^{2}\right)  -1\right)  \right)  dx\\
&\ \   +\frac{\lambda}{4}\int_{\mathbb{R}^{4}}\left(  \exp\left(  2\left\vert
u_{j}^{\prime}\right\vert ^{2}\right)  2\left\vert u_{j}^{\prime}\right\vert
^{2}-\left(  \exp\left(  2\left\vert u_{j}^{\prime}\right\vert ^{2}\right)
-1\right)  \right)  dx\\
&  \leq I_{V}\left(  u_{j}\right)  +o_{j}\left(  1\right).
\end{align*}
Let $j\rightarrow \infty$, we derive $m_{\infty}\leq m_{V}$, which is a contradiction. This accomplishes the proof of Lemma \ref{adxin1}.
\end{proof}

\bigskip

\bigskip

Lu Chen

School of Mathematics and Statistics

Beijing Institute of Technology

Beijing 100081, P. R. China

chenlu5818804@163.com

\bigskip

Guozhen Lu

Department of Mathematics

University of Connecticut

Storrs, CT 06269, USA

E-mail: guozhen.lu@uconn.edu

\bigskip

Maochun Zhu

Faculty of Science

Jiangsu University

Zhenjiang, 212013, P. R. China

zhumaochun2006@126.com

\bigskip

\bigskip

\end{document}